\documentclass[10pt,a4paper]{article}
\usepackage[utf8]{inputenc}
\usepackage[colorlinks=true]{hyperref}
\usepackage{amsmath, amsthm}
\usepackage{amsfonts}
\usepackage{amssymb}
\usepackage{fixltx2e}
\usepackage{mathtools}
\usepackage[hyperref, dvipsnames]{xcolor}
\hypersetup{
  colorlinks=true,
}
\usepackage{enumerate}

\usepackage[left=2cm,right=2cm,top=2cm,bottom=2cm]{geometry}
\usepackage{comment}
\newcommand{\grad}{\nabla}

\newcommand{\norm}[2]{\lVert #1\rVert_{#2}}
\newcommand{\weaklyto}{\rightharpoonup}
\newcommand{\weaklystar}{\overset{\ast}{\rightharpoonup}}

\DeclareMathOperator*{\esssup}{ess\,sup}
\DeclareMathOperator*{\essinf}{ess\,inf}

\newcommand{\cts}{\hookrightarrow}
\newcommand{\loc}{\mathrm{loc}}
\newcommand{\ctsCompact}{\xhookrightarrow{c}}

\newtheorem{theorem}{Theorem}[section]
\newtheorem{prop}[theorem]{Proposition}
\newtheorem{lem}[theorem]{Lemma}

\newtheorem{remark}[theorem]{Remark}

\theoremstyle{definition}

\title{Analysis of a quasi-variational contact problem arising in thermoelasticity\footnotetext{A. A. was partially supported by the DFG through the DFG SPP 1962 Priority Programme \emph{Non-smooth and Complementarity-based
Distributed Parameter Systems: Simulation and Hierarchical Optimization} within project 10. C. N. R. was supported by NSF grant DMS-2012391, and  acknowledges the support of Germany's Excellence Strategy - The Berlin Mathematics Research Center MATH+ (EXC-2046/1, project ID: 390685689) within project AA4-3.}  }
\author{Amal Alphonse\thanks{Weierstrass Institute, Mohrenstrasse 39, 10117 Berlin, Germany ({\tt alphonse@wias-berlin.de})} \and Carlos N. Rautenberg\thanks{Department of Mathematical Sciences and the Center for Mathematics and Artificial Intelligence (CMAI), George Mason University, Fairfax, VA 22030, USA ({\tt crautenb@gmu.edu})
} \and  Jos\'e Francisco Rodrigues\thanks{CMAFcIO - Departamento  de  Matem\'{a}tica,  Faculdade  de  Ci\^{e}ncias,  Universidade  de Lisboa P-1749-016 Lisboa, Portugal ({\tt jfrodrigues@ciencias.ulisboa.pt}) }  }
\begin{document}
\hypersetup{
  urlcolor     = blue, 
  linkcolor    = Bittersweet, 
  citecolor   = Cerulean
}
\maketitle
\begin{abstract}
We formulate and study two mathematical models of a thermoforming process involving a membrane and a mould as implicit obstacle problems. In particular, the membrane-mould coupling is determined by the thermal displacement of the mould that depends in turn on the membrane through the contact region. The two models considered are a stationary (or elliptic) model and an evolutionary (or quasistatic) one.
For the first model, we prove the existence of weak solutions  by solving an elliptic quasi-variational inequality coupled to elliptic equations. By exploring the fine properties of the variation of the contact set under non-degenerate data, we give sufficient conditions for the existence of regular solutions, and under certain contraction conditions, also a uniqueness result. We apply these results to a series of semi-discretised problems that arise as approximations of regular solutions for the evolutionary or quasistatic problem. Here, under certain conditions, we are able to prove existence for the evolutionary problem and for a special case, also the uniqueness of time-dependent solutions.
\end{abstract}
{
  \hypersetup{linkcolor=black}
  \tableofcontents
}
\section{Introduction} 
In this paper, we propose a model of a thermoforming process involving a system of elliptic-parabolic partial differential equations with an implicit obstacle constraint that describes the thermoelastic behaviour of materials in the process. The obstacle is \emph{a priori} unknown and it depends on the other unknown variables in the system leading to  a problem that is \emph{quasi-variational} in nature. We study both the evolutionary model as well as its associated stationary counterpart which is a system of elliptic partial differential equations coupled to a variational inequality. The content of the paper is on the mathematical analysis of these two systems: we study the issues of existence, uniqueness, (local) regularity and other qualitative properties of solutions.

A variety of industrial processes for the manufacturing of precision parts entail forcing a sheet or membrane of a specific polymer onto a mould by means of positive or negative air pressure (or other mechanisms). In order to enter the shape-acquiring phase and to reduce brittleness of the material, the sheet is heated to an easy-to-deform state while the mould is not; in fact, the mould might be cooled down (this could be done only in particular regions to control the thickness of the material in the final piece). The temperature difference triggers a complex heat transfer process during contact, and this is  further coupled with changes of shape of the mould and sheet due to the thermal linear expansion phenomenon. Finally, the membrane acquires the shape of the mould via a cooling down phase. One common process of this type is \emph{thermoforming} which involves the manufacturing of plastic components on a wide range of products (and scales) that go from car panels ($\simeq$ 1 meter) to microfluidic structures ($\simeq$ 1 micrometer).

\begin{figure}[htb]
\centering
\includegraphics[scale=0.5]{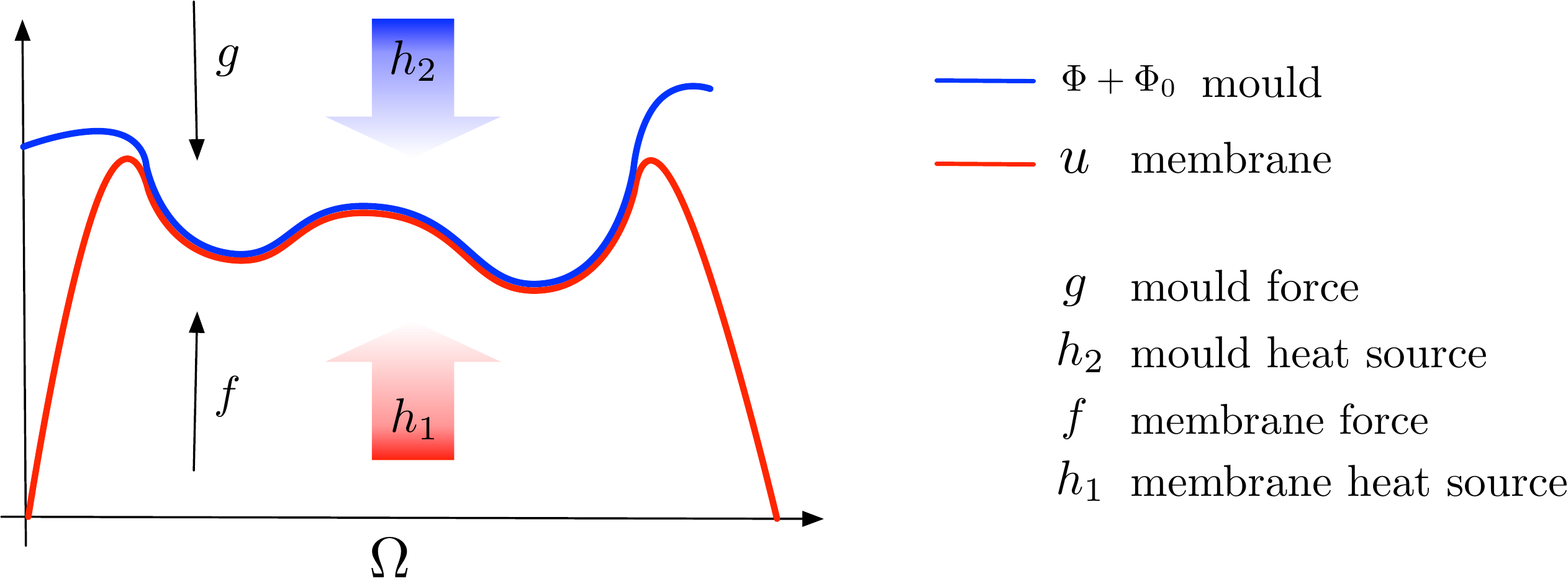}
\caption{Contact of the membrane and the mould.}
\label{fig:diagram}
\end{figure}

The problem described above is a highly complex type of \emph{contact problem in thermoelasticity} \cite{nowacki1975dynamic,nowacki2013thermoelasticity,parkus2012thermoelasticity} as it couples elastic with heat transfer phenomena and at the same time the constraint associated to the non-penetration condition between the membrane and the mould holds. In particular, the static conduction of heat across the two thermoelastic materials depends on the extent of the contact area which in turn depends  on the thermoelastic displacement. The coupling of the three processes (\emph{elasticity, heat transfer}, and \emph{contact}) leads to a variety of theoretical and numerical difficulties that are yet to be resolved in a general setting. Let us mention a few of these issues that shed light on the rich variety of physical phenomena that arise, none of which occur in the absence of contact. The mathematical formulation of the problem requires a proper solution concept that takes into account the geometrical constraint of non-penetration, but also provides an expression of how the force is transferred between the materials. The non-penetration condition is usually resolved locally by assuming convexity by means of  a so-called \emph{gap function} \cite{wriggers2004computational}, and the expression of the contact pressure can be described  via a power of this gap function with experimentally measured parameters \cite{wriggers2004computational,kragelsky2013friction}. The consideration of friction on the contact leads to even more complex settings because the friction phenomena induces heat in many cases \cite{andersson2005existence,MR3241640}. We further refer the reader  for existence of solutions, their uniqueness or non-uniqueness, and overall modeling of similar problems  to \cite{andersson2005existence,barber1999thermoelasticity,copetti1993one,duvaut1979free,rivera1998multidimensional,shi1990uniqueness},  and \cite{Munro2001,Warby2003,karamanou2006computational,zavarise1992real,wriggers2004computational,wriggers2008formulation} for numerical approximations and computational aspects.

There are four main sources of difficulties and features  associated to the type of processes mentioned above (and that are present in our problem of interest) that we are required to capture mathematically:\vspace{-.1cm}
  \begin{enumerate}[(i)]
 	\item contact between the mould and the sheet/membrane is frictionless and occurs on large regions\vspace{-.2cm}
 	\item heat transfer occurs mainly when contact occurs\vspace{-.2cm}
	\item the elastic properties of the membrane are temperature-dependent\vspace{-.2cm}
 	\item the thermal expansion of the mould might generate non-negligible effects in the finished parts.
 \end{enumerate}
In mathematical terms,  concerning (i) and initially disregarding temperature, the contact problem associated to the heated  membrane/sheet and the mould is assumed to be modeled as a variational inequality under the assumption of perfect sliding, i.e., no friction is present. Item (ii) implies that heat equations contain terms involving active contact regions. Item (iii) implies that the Lam\'{e} coefficients of the plastic sheet depends on temperature and hence induces in the displacement equations a second-order differential operator that depends on temperature as well. Finally, (iv) specifies that spatial differences in (or gradients of) the temperature  should be considered in the displacement equation of the mould in order to capture the possible expansion.

We discuss now the assumptions that lead to our stationary and evolutionary model.  We assume that displacements of the membrane and the mould occur only in one spatial direction and denote those displacements as $u$ and $\Phi$, respectively. The initial or undeformed membrane and mould are denoted respectively by $u_0\equiv 0$ and $\Phi_0$; hence the deformed structures are given by $u$ and $\Phi+\Phi_0$, and the former is assumed to be below the latter. Additionally, we do not allow for displacement in the boundary $\partial \Omega$ and we assume that mechanical contact has a negligible effect on the deformation of the mould. We take for granted that heat transfer between the membrane and mould is present only when contact occurs and that boundaries are insulated; we use $\theta_1$ and $\theta_2$ to denote the temperatures of the membrane and the mould, respectively. Furthermore, the linear thermal expansion effect on the mould appears in the displacement equation as a term proportional to the temperature difference $\theta_1-\theta_2$ and is active only when contact is present. As we later explain, the difference $\theta_1-\theta_2$ plays the role of the spatial gradient $\nabla \theta$ in the general three-dimensional thermoelastic setting. Additionally, we do not assume thermal expansion of the membrane due to the fact that this term is significantly smaller than the force pushing the membrane. From this, one directly observes that the coupling depicted in Figure \ref{fig:diagram} leads to the dependence $u\mapsto \Phi(u)+\Phi_0$ and hence to the constraint
\begin{equation*}
u\leq \Phi(u)+\Phi_0,
\end{equation*}
thus determining the quasi-variational inequality formulation for the problem of interest as the upper obstacle constraint depends on $u$. 

In this paper, we study a static and a quasistatic model associated to the thermoelastic contact problem described above. Without loss of generality, when $\Phi_0$ is smooth, we can assume\footnote{This follows by the transformation $g \mapsto g-\Delta\Phi_0$ in the system \eqref{eq:equation} that will shortly be introduced.}
\[\Phi_0 \equiv 0.\]
In order to simplify the upcoming presentation, we use the notation
\[A_\theta u := -\grad \cdot (a(\theta_1)\grad u).\]
We first study the stationary or static case. In order to keep the problem tractable, we suppose that the membrane and the mould  have zero thickness and can be described by functions on a domain $\Omega\subset \mathbb{R}^n$. 
We assume throughout the paper that
\[\Omega \text{ is a bounded Lipschitz domain}\]
since our method requires in particular the compactness of the embedding $H^1(\Omega) \ctsCompact L^2(\Omega)$ (as a matter of notation, $X \cts Y$ means that $X$ is continuously embedded in $Y$ and $X \ctsCompact Y$ stands for a compact embedding).
Although the physically relevant case is dimension $n=2$, the analysis in this paper is still valid for general dimensions and therefore we include it as a possibility. The mathematical formulation reduces to the following free boundary problem:
\begin{subequations}\label{eq:equation}
\begin{align}
&\text{for $i=1, 2$:}  &- \kappa_i \Delta \theta_i + c_i\theta_i &= h_i + (-1)^ib_i(\theta_1-\theta_2)\chi_{\{u=\Phi\}} &&\text{in $\Omega$},\label{eq:equation1}\\
& &\partial_n \theta_i &= 0 &&\text{on $\partial\Omega$},\label{eq:equation2}\\
&&-\Delta \Phi &= \alpha(\theta_1 - \theta_2)\chi_{\{u=\Phi\}}+g &&\text{in $\Omega$},\label{eq:equation5}\\
&& \Phi &= 0 &&\text{on $\partial\Omega$},\label{eq:equation6}\\
&&\qquad u \leq \Phi,\quad A_\theta u&\leq  f, \quad (A_\theta u-f)(u-\Phi) = 0 &&\text{in $\Omega$}, \label{eq:equation3}\\
&& u &= 0 &&\text{on $\partial\Omega$}\label{eq:equation4}.
\end{align}
\end{subequations}
Here,
$\kappa_i>0$, $c_i> 0$, $f, g, h_i\colon \Omega\to \mathbb{R}$, $b_i\geq 0$ and $\alpha > 0$ for $i=1,2$ are the given data. Additionally, $\chi_{\{u=\Phi\}}$ is the characteristic function of the contact set and $\partial_n\theta$ denotes the normal derivative at the boundary $\partial\Omega$ of $\Omega$. We assume throughout that
\begin{equation}\label{ass:onCoefficientFn}
	a \in C^0(\mathbb{R}) \quad \text{ with } \quad 0 < \lambda_1 \leq a \leq \lambda_2
\end{equation}
and on occasion
\begin{equation}
a \in C^1(\mathbb{R}) \text{ with $a'$ bounded}.\tag{2'}\label{ass:onCoefficientDeriv}
\end{equation}
We take the system \eqref{eq:equation} to hold in the sense of distributions. Observe that it is a free boundary problem since the boundary of the contact set $\{u=\Phi\}$ is not known \emph{a priori}. A schematic (based on \cite{barber1999thermoelasticity}) of interactions among phenomena based on the system \eqref{eq:equation} can be illustrated as in Figure \ref{fig:diagramInt}.

%

\begin{figure}[htb]
\centering
\includegraphics[scale=0.5]{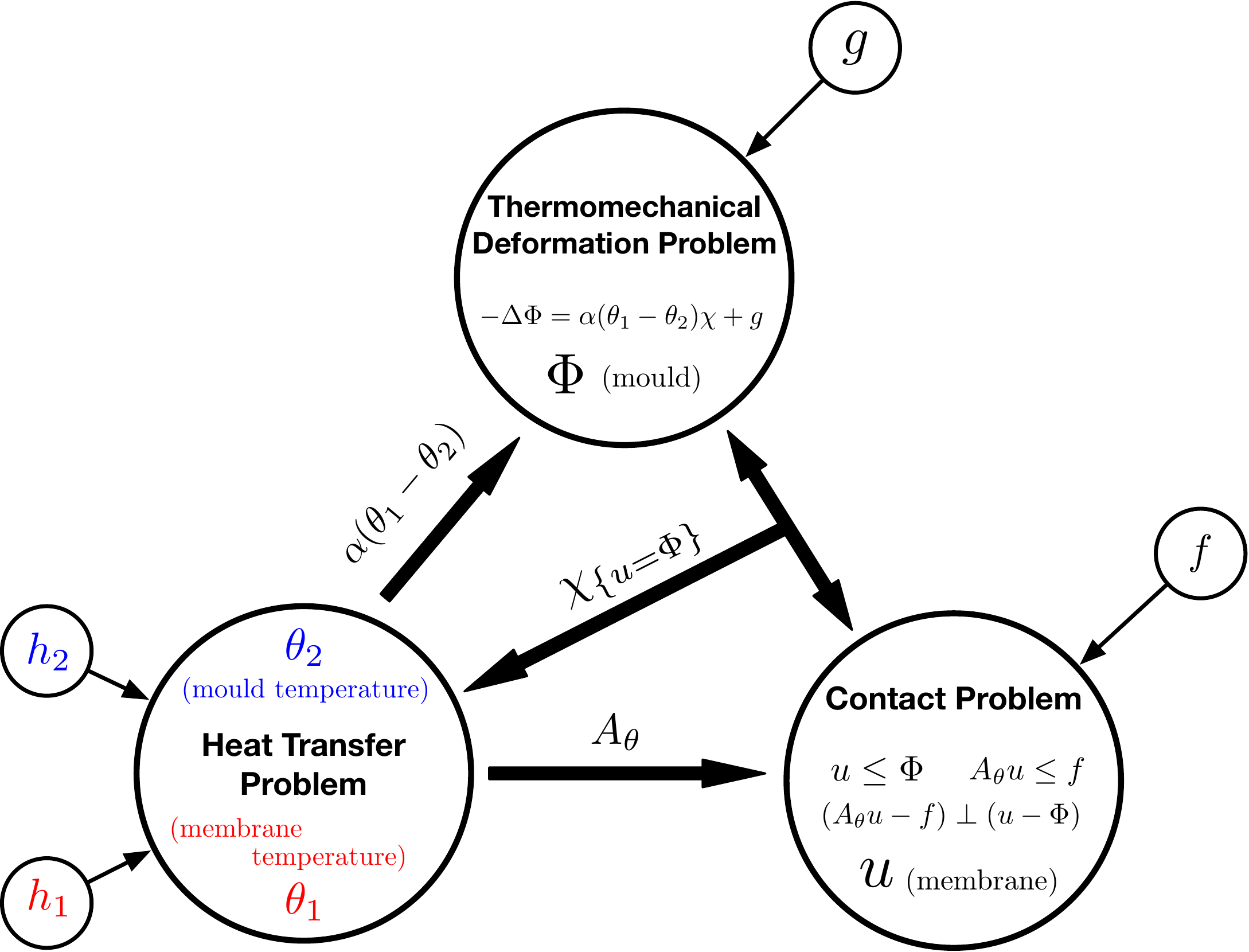}
\caption{Coupling between the main phenomena of the problem of interest leading to the quasi-variational formulation.}
\label{fig:diagramInt}
\end{figure}

A few words are in order concerning \eqref{eq:equation}. In the right-hand side of \eqref{eq:equation1} we observe that upon contact $u=\Phi$, heat flows in the direction of lower temperature as expected. Additionally, a heat conservation law may be derived: summing the equations for $\theta_i$ and integrating, we get 
\[\int_{\Omega} (c_1\theta_1 + c_2\theta_2) = \int_\Omega (h_1+h_2) + (b_2-b_1)\int_{\{u=\Phi\}}(\theta_1-\theta_2).\]
This shows that, in equilibrium, the total heat amount accounts for the sources together with the heat exchange  between the membrane and the mould on the contact set.

For fixed $\theta_1,\theta_2$, and $\Phi$, \eqref{eq:equation3}--\eqref{eq:equation4} correspond to a variational inequality capturing the behaviour of an elastic membrane reacting to a forcing term $f\colon\Omega \to \mathbb{R}$ and with obstacle $\Phi$ (see for instance \cite{Rodrigues} and its references). Finally, \eqref{eq:equation5}--\eqref{eq:equation6} describes displacements of the mould due to linear thermal expansion when contact occurs. The term $\alpha(\theta_1 - \theta_2)\chi_{\{u=\Phi\}}$ originates from the general equations of thermoelasticity in which case the body force associated to thermal expansion is proportional to the spatial gradient of the temperature; since we assume only displacement in the vertical coordinate, the gradient of the temperature term becomes proportional to the temperature difference $\theta_1-\theta_2$. In addition, $g\colon\Omega\to \mathbb{R}$ is a general term utilised in order to reduce the entire  system \eqref{eq:equation1}--\eqref{eq:equation4} to the case $\Phi_0=0$, but it may also take into account the external forces on the mould.

Problem \eqref{eq:equation1}--\eqref{eq:equation4} is a quasi-variational inequality or an implicit obstacle type problem: if we consider the equations for $\theta_1$, $\theta_2$, and $\Phi$, then, given an arbitrary $u$ and provided there exists a solution to the aforementioned system, $\Phi$ is a function of $u$, i.e., $\Phi=\Phi(u)$, and hence the variational inequality problem involving $u$ contains an obstacle $\Phi$ that in turn also depends on $u$. This implies that \eqref{eq:equation1}--\eqref{eq:equation4} is a highly nonconvex nonlinear problem for which existence of solutions is not immediately clear. This is studied in \S \ref{sec:elliptic}. To facilitate the study of \eqref{eq:equation} in a convenient fashion, we develop some theory for the weakly coupled auxiliary problem related to the temperatures: given $\sigma \in L^\infty(\Omega)$ with $\sigma \geq 0$,  consider
\begin{equation*}
\begin{aligned}
\text{for $i=1, 2$:}\qquad - \kappa_i \Delta \vartheta_i  + c_i\vartheta_i &= h_i + (-1)^ib_i(\vartheta_1-\vartheta_2)\sigma&&\text{in $\Omega$},\\
\partial_n \vartheta_i &= 0&&\text{on $\partial\Omega$}.
\end{aligned}
\end{equation*}
In \S \ref{sec:auxProblem}, we establish a number of properties for this system, which has an independent interest outside of the topic in consideration in this paper.

The second problem we turn our interest to corresponds to the following evolutionary version of \eqref{eq:equation}:
\begin{subequations}\label{eq:parabolicEquation}
\begin{align}
&\text{for $i=1, 2$:} &\partial_t \theta_i - \kappa_i \Delta \theta_i + c_i\theta_i &= h_i + (-1)^ib_i(\theta_1-\theta_2)\chi_{\{u=\Phi\}} &&\text{in $Q$},\label{eq:parabolicEquation1}\\
&&\partial_n \theta_i &= 0 &&\text{on $\Sigma$},\label{eq:parabolicEquation2}\\
&& \theta_i(0) &= \theta_{i0} &&\text{in $\Omega$},\label{eq:parabolicEquation3}\\
\nonumber &\text{for a.e. $t \in (0,T)$:}\\
&&-\Delta \Phi(t) = \alpha(&\theta_1(t) - \theta_2(t))\chi_{\{u(t)=\Phi(t)\}} +g(t) &&\text{in $\Omega$},\\
&&\Phi(t) &= 0 &&\text{on $\partial\Omega$},\\
&&u(t) \leq \Phi(t),\quad A_\theta(t) u(t)&\leq  f(t), \quad (A_\theta(t) u(t)-f(t))(u(t)-\Phi(t)) = 0 &&\text{in $\Omega$},\\
&&u(t) &=0 &&\text{on $\partial\Omega$},
\end{align}
\end{subequations}
where $Q:=(0,T)\times\Omega$, $\Sigma := (0,T)\times \partial\Omega$, and $\theta_{10}, \theta_{20}\colon \Omega\to\mathbb{R}$ are given initial data and we have assumed that the inertial contributions for the elastic membrane are negligible. We partition the time interval into $N$ uniform subintervals and hence consider a time-discrete quasistatic QVI. 
In \S \ref{sec:parabolic} we study existence of solutions and other properties for \eqref{eq:parabolicEquation}. In it, we will see that we need the strong assumption $a' \equiv 0$ to obtain existence for \eqref{eq:parabolicEquation}, though it is not necessary for well posedness of the semi-discretised problem.

\paragraph{Related work.} Some of the earliest mathematical results on quasistatic contact problems were obtained in  \cite{Andersson1991, MR1169542, MR1433753, MR750004}. We also refer the reader to the survey paper \cite{MR1835153} and book \cite{MR3241640} for further details and references. The works that are most closely related to this paper are the following. In \cite{Rodrigues2000}, the situation is of an elastic membrane stretched over a (rigid) obstacle which is influenced by a temperature field that itself depends on the contact between the membrane and the obstacle (in the same fashion as in our model, through a characteristic function of the coincidence set).  The equilibrium problem is then studied, which leads to a variational inequality (for the membrane) coupled to a PDE for the temperature of the membrane. In \cite{Rodrigues2004},  the problem under consideration models the diffusion and absorption of oxygen in tissue, with the free boundary being the separation between the regions where there is oxygen and where there is no oxygen. The authors consider a parabolic variational inequality describing the concentration of oxygen with a diffusion coefficient depending on the temperature  which itself  satisfies a heat equation with a source term that includes a characteristic function of the set where the concentration is strictly positive.  We also mention \cite{Rodrigues2005} where an elastic membrane is deformed over a hot plane and the resulting system is analysed by developing useful properties of the obstacle problem, and a simpler model treated as an elliptic quasi-variational inequality in \cite{MR3903798}, also motivated by thermoforming, in which a different analysis and some numerical simulations are presented.

\paragraph{Organisation of the paper.} This work is divided into two parts: \S \ref{sec:elliptic} treats the elliptic problem \eqref{eq:equation} and \S \ref{sec:parabolic} the quasistatic problem \eqref{eq:parabolicEquation}, and both start by the statement of the main results.

First, we shall prove existence of a weak solution to \eqref{eq:equation} in \S \ref{sec:regFunctions} -- \S \ref{sec:passageLimit} where we do not obtain necessarily the characteristic function $\chi_{\{\Phi=u\}}$ but a weaker function $\chi$ that satisfies
\[0 \leq \chi \leq \chi_{\{\Phi=u\}} \leq 1 \quad \text{a.e. in $\Omega$,}\]
for which a physical interpretation is provided below.  We study in \S \ref{sec:auxProblem} some interesting and useful properties of the weakly coupled elliptic system \eqref{eq:equation1}--\eqref{eq:equation2} in a slightly more general setting, obtaining, in particular, some $L^\infty$ and comparison properties of the temperatures. In \S \ref{sec:regFunctions}, returning to \eqref{eq:equation}, we consider an approximated problem by regularising the Heaviside function and obtaining its solution by a fixed point argument. Existence of weak solutions for the elliptic problem is obtained in \S \ref{sec:passageLimit} by passing to the limit in the regularisation parameter. A local regularity for the obstacle problem, given in \S \ref{sec:localRegularity}, allows the identification of the characteristic function  which gives a regular solution under compatibility assumptions in \S \ref{sec:ellipticIdChar}. The section ends with a uniqueness result in \S \ref{sec:uniquenessElliptic} by the contraction principle and using a sharp $L^1$-continuous dependence estimate on the characteristic functions of the contact sets of the obstacle problem \cite{Rodrigues}.

In \S \ref{sec:parabolicMain} we state the main results only for regular solutions of the evolutionary problem \eqref{eq:parabolicEquation}, i.e., we directly obtain $\chi = \chi_{\{\Phi=u\}}$. We are able to show the existence and continuity in time under the restriction of the coefficient $a$ being constant (i.e., $a' \equiv 0$; the general case remains an open problem). We also have the uniqueness of solutions under an additional non-degeneracy assumption, which is implied by a sufficiently large force $f$. In \S \ref{sec:parabolicSemi}, we perform a semi-discretisation in time, leading to a series of elliptic problems at each iteration step $N$ for which we can apply the existence results for weak and regular solutions of \S \ref{sec:elliptic} under the general assumptions \eqref{ass:onCoefficientFn} and \eqref{ass:onCoefficientDeriv}  on the coefficient $a=a(\theta_1)$. The semi-discretisation is useful also for the purposes of numerical simulation which we leave for future work. The study of interpolants in \S \ref{sec:interpolants} (see also Appendix \ref{sec:app}) and obtainment of appropriate \emph{a priori} estimates allows the passage to the limit as $N \to \infty$ to obtain a weak solution of the evolution problem in \S \ref{sec:limitingBehaviour}.  A careful identification under appropriate conditions of the time-dependent characteristic function in \S \ref{sec:parabolicIdChar}, and Bochner and pointwise limits in \S \ref{sec:concExistence}, shows that in fact the weak solution is a regular one. This implies also that $\chi=\chi_{\{\Phi=u\}} \in C^0((0,T);L^p(\Omega))$ for all $p < \infty$ as shown in \S \ref{sec:parCont} as a consequence of the strong continuous dependence of the characteristic function of the contact set in the elliptic obstacle problem \cite{Rodrigues} as well as the uniqueness of regular solutions under similar contraction conditions in \S \ref{sec:uniquenessParabolic}. Finally, we conclude in \S \ref{sec:remarksOnNonRegularCase} with some remarks on a general situation when the identification of the functions $\chi^k$ obtained from the semi-discretised problem fail and lead to a limit degenerate evolution case, for which a very weak formulation is introduced.

\section{The elliptic (stationary) problem}\label{sec:elliptic}
The existence of solutions for \eqref{eq:equation} will be proved in several stages: we first regularise the characteristic function that appears in the system and then prove existence for the resulting system using a fixed point approach. Then we pass to the limit in the regularisation parameter to obtain a slightly weaker system than \eqref{eq:equation}. Finally, we will prove
a regularity result under a non-degeneracy assumption on the forcing term which then will yield existence for \eqref{eq:equation}. 

Let us begin by recalling the definition of the Heaviside graph
\begin{equation*}
H(s) := \begin{cases}
0 &: s < 0\\
[0,1] &: s = 0\\
1 &: s > 0.
\end{cases}
\end{equation*}
Observe that for any $u$ and $\Phi$ solving \eqref{eq:equation},  a.e. in $\Omega$,
\[\chi_{\{u=\Phi\}} \geq 1-H(\Phi-u) \quad\text{and} \quad \chi_{\{u=\Phi\}} \in 1-H(\Phi-u).\]
In view of this, we say that $(\theta_1, \theta_2, u, \Phi) \in H^1(\Omega)^2 \times (H_0^1(\Omega))^2$ is a \textbf{weak solution} of \eqref{eq:equation} if there exists $\chi \in 1-H(\Phi-u)$ such that
\begin{equation}\label{eq:weakFormWeakSoln}
\begin{aligned}
 \int_\Omega \kappa_1\grad \theta_1 \cdot \grad \eta + c_1 \theta_1\eta &= \int_\Omega (h_1 -b_1(\theta_1-\theta_2)\chi)\eta&&\forall \eta \in H^1(\Omega),\\
\int_\Omega \kappa_2 \grad \theta_2 \cdot \grad \zeta + c_2\theta_2\zeta &= \int_\Omega (h_2 + b_2(\theta_1-\theta_2)\chi)\zeta &&\forall \zeta \in H^1(\Omega),\\
\int_\Omega \grad \Phi  \cdot \grad \xi &= \alpha\int_\Omega ((\theta_1 - \theta_2)\chi+ g)\xi &&\forall \xi \in H^1_0(\Omega),\\
u \in \mathbb{K}(\Phi) : \int_\Omega a(\theta_1)\grad u \cdot \grad (u -v) &\leq \int_\Omega f(u-v) &&\forall v \in \mathbb{K}(\Phi),
\end{aligned}
\end{equation}
where for a function $\phi \colon \Omega \to \mathbb{R},$ the set $\mathbb{K}(\phi)$ is defined by
\begin{equation*}
\mathbb{K}(\phi) := \{ \varphi \in H^1_0(\Omega) : \varphi \leq \phi \text{ a.e. in $\Omega$}\}.
\end{equation*}
The weak solution has the following physical interpretation. When there is no contact between the mould and the membrane, one has $\chi \equiv 0$, which agrees with the behaviour of the term $\chi_{\{u=\Phi\}}$ appearing in \eqref{eq:equation}. In case of contact, one has $\chi \leq 1$, which means that at least a fraction of the expected heat exchange occurs, leading to a deformation of the mould under contact. If $\chi =1$ in the contact region, we have $\chi=\chi_{\{\Phi=u\}}$ as expected in the case of regular solutions.

We now state and discuss the main results that we are able to prove.
\subsection{Main results}
We first of all study the existence of weak solutions as formulated above. The proof of the next theorem will be conducted from Sections \ref{sec:regFunctions} to  \ref{sec:localRegularity}. 
\begin{theorem}[\textsc{Existence and regularity of weak solutions}]\label{thm:existence}
Suppose that $f, g, h_1,h_2 \in L^2(\Omega)$ and that
\begin{equation}\label{ass:onConstants}
c_0 := \min\left(c_1 - \frac{(b_2 - b_1)^+}{4}, c_2 - \frac{(b_1 - b_2)^+}{4}\right) > 0.
\end{equation}
Then \eqref{eq:equation} has a weak solution
\[(\theta_1, \theta_2, u, \Phi) \in (H^1(\Omega)\cap H^2_{\loc}(\Omega))^2 \times H^1_0(\Omega) \times (H^1_0(\Omega) \cap H^2_{\loc}(\Omega)) \quad\text{with}\quad \chi \in 1-H(\Phi-u)\]
satisfying \eqref{eq:weakFormWeakSoln}. 

\medskip

\noindent Furthermore, under \eqref{ass:onCoefficientDeriv}, 
\begin{enumerate}[(i)]
\item if $n \leq 3$ or $a' \equiv 0$, then $u \in H^2_{\loc}(\Omega)$,
\item if  $h_1, h_2 \in L^p(\Omega)$ for $p > n$, then $u \in H^2_{\loc}(\Omega)$ and $\theta_1, \theta_2 \in C^{0,\gamma}(\bar\Omega) \cap  C^{1,\mu}(\Omega)$ for some $\gamma, \mu\in (0,1)$,
\item if  $f, g, h_1, h_2 \in L^p_\loc(\Omega)$ for $p > n$, then $\theta_1, \theta_2, \Phi, u \in W^{2,p}_{\loc}(\Omega) \cap C^{1,\mu}(\Omega)$ for some $\mu \in (0,1)$.
\end{enumerate}
\end{theorem}
{Let us note that if we assume that $\Omega$ is either convex or a $C^{1,1}$ domain, we can use elliptic regularity results such as \cite{MR2333653} (for $\Omega$ convex) and \cite[Lemma 9.17]{Gilbarg} (for $\Omega \in C^{1,1}$), and the above regularity results hold not just locally but globally on $\Omega$.} 

It is useful to define
\begin{align}\label{eq:mAndM}
m:= \min\left(\frac{1}{c_1}\essinf_\Omega h_1,  \frac{1}{c_2}\essinf_\Omega h_2\right)\quad\text{and}\quad M:= \max\left(\frac{1}{c_1}\esssup_\Omega h_1, \frac{1}{c_2}\esssup_\Omega h_2\right).
\end{align}
In \S \ref{sec:auxProblem}, we will prove a result which in particular, implies the non-negativity of solutions for signed data as well as a bound on the difference of the temperatures:
\begin{align}
\text{if $h_1 \geq 0$,  $h_2 \geq 0$ and \eqref{ass:onConstants} holds, then $\theta_1 \geq 0$ and $\theta_2 \geq 0$,}\label{eq:nonnegativityOfTemps}\\
\text{if $h_1, h_2 \in L^\infty(\Omega)$ and \eqref{ass:onConstants} holds, then } \norm{\theta_1-\theta_2}{L^\infty(\Omega)} \leq M-m.\label{eq:LinftyEstimateTempDiff}
\end{align}
We say that $(\theta_1, \theta_2, u, \Phi, \chi)$ is a \textbf{regular solution} of \eqref{eq:equation} if it is a weak solution that in addition satisfies $(\theta_1, \theta_2, \Phi, u) \in H^2_{\loc}(\Omega)^4$ and $\chi=\chi_{\{\Phi=u\}}$. A few technical assumptions are needed to guarantee the existence of regular solutions. Namely, these include 
\begin{align}
&\frac{c_2}{\kappa_2} \geq \frac{c_1}{\kappa_1} \qquad\text{and}\qquad \frac{h_1}{\kappa_1} \geq \frac{h_2}{\kappa_2} \geq 0\label{ass:onConstantsForRegularity}
\end{align}
and also
\begin{align}
&\text{if  $n > 3$, $h_1, h_2 \in L^p(\Omega)$ for $p>n$}\label{ass:onHeats}.
\end{align}
The second assumption guarantees that $u \in H^2_{\loc}(\Omega)$ (by Theorem \ref{thm:existence}) and the first one essentially ensures that $\theta_1-\theta_2 \geq 0$ (we shall prove this in Proposition \ref{prop:comparisonPrinciple}) which is needed for technical reasons\footnote{In fact it also implies non-negativity of the temperatures; it would be sufficient to have $\theta_1 \geq \theta_2$ instead of \eqref{ass:onConstantsForRegularity}.}. 

\begin{theorem}[\textsc{Existence of regular solutions}]\label{thm:existenceRegularity}
Under the assumptions of Theorem \ref{thm:existence}, if furthermore \eqref{ass:onCoefficientDeriv}, \eqref{ass:onConstantsForRegularity} and \eqref{ass:onHeats} hold,  and 
\begin{align}
&f+\grad \cdot (a(\theta_1)\grad \Phi) > 0 \quad \text{a.e. in $\Omega$,}
\label{ass:forRegularity}
\end{align}
then 
\[\chi = \chi_{\{u=\Phi\}}\]
and there exists a regular solution to \eqref{eq:equation}.
\end{theorem}
This theorem essentially states that, under some conditions on the data and if the forcing term $f$ is large enough, then we can identify the function $\chi$ of Theorem \ref{thm:existence} as the characteristic function of the contact region. Roughly speaking, ``strong forces give rise to regular solutions". 

We give now some conditions on the data that can be checked \emph{a priori} ensuring that \eqref{ass:forRegularity} is met. These conditions are particularly easy to check when the coefficient function $a$ is a constant (see Remark \ref{rem:1} for more details). In case $a' \not\equiv 0$, one needs $\grad \theta_1\cdot \grad \Phi \in L^\infty(\Omega)$. Let us describe a situation where this bound holds. Suppose that
\begin{equation}\label{eq:assForGradBd}
\text{$h_1, h_2,$ and  $g \in L^p(\Omega)$ for $p> n$ and $\Omega$ is convex or $C^{1,1}$.}
\end{equation}
The increased regularity of the domain implies $W^{2,p}(\Omega)$ (global) regularity for the solutions. Applying \cite[Thereom 6.1, \S 7]{LSUElliptic}, $\theta_1 \in W^{2,p}(\Omega) \cts W^{1,\infty}(\Omega)$. A similar argument gives the same global regularity for $\Phi$ too and, in combination, we have the existence of a constant $C_{\mathrm{grad}}$ such that
\begin{equation}\label{eq:boundOnGradients}
\norm{\grad \theta_1\cdot \grad \Phi}{L^\infty(\Omega)} \leq C_{\mathrm{grad}}.
\end{equation}

\begin{prop}\label{prop:conditionForRegularity}
Assume  \eqref{ass:onCoefficientDeriv}, \eqref{ass:onConstants}, \eqref{ass:onConstantsForRegularity}, $h_1, h_2 \in L^\infty(\Omega)$ and if $a' \not\equiv  0$, assume also \eqref{eq:assForGradBd}, and define 
\begin{equation}\label{ass:M2}
K_{\mathrm{grad}}:= \begin{cases}
0 &: \text{if $a' \equiv 0$}\\
C_{\mathrm{grad}} &: \text{otherwise}.
\end{cases}
\end{equation}
Then 
\begin{equation}\label{ass:onSourceTermsForRegularity}
f-\lambda_2g^+ + \lambda_1g^- >  \alpha\lambda_2(M-m) + \norm{a'}{\infty}K_{\mathrm{grad}}
\end{equation}
implies that assumption \eqref{ass:forRegularity} holds.
\end{prop}
\begin{proof}
By splitting $g$ into its positive and negative parts and using the boundedness assumption on $a$ and the bounds \eqref{eq:LinftyEstimateTempDiff} and \eqref{eq:boundOnGradients},
\begin{align*}
f+ \grad\cdot(a(\theta_1)\grad\Phi) &= f-a(\theta_1)g-\alpha a(\theta_1)(\theta_1-\theta_2)\chi+a'(\theta_1)\grad\theta_1\grad\Phi\\
&\geq f-\lambda_2g^+ +  \lambda_1g^- - \alpha\lambda_2\norm{\theta_1-\theta_2}{L^\infty(\Omega)} - \norm{a'}{\infty}\norm{\grad \theta_1\grad \Phi}{L^\infty(\Omega)}\\
&\geq f-\lambda_2g^+ +  \lambda_1g^- - \alpha\lambda_2(M-m)- \norm{a'}{\infty}K_{\mathrm{grad}}.
\end{align*}
\end{proof}
\begin{remark}\label{rem:1}
Observe that in the case where the coefficient function $a$ is a constant, \eqref{ass:M2} is unnecessary whilst \eqref{ass:onSourceTermsForRegularity} is greatly simplified:  we merely need
\begin{align*}
f-ag &> a\alpha(M-m).
\end{align*}
On the other hand, if $a' \not\equiv 0$ and if $\theta_1$ and $\Phi$ are only $W^{2,p}_\loc(\Omega)$ for $p >n$, we can ask for the estimate \eqref{eq:boundOnGradients} locally, i.e., with $C'_{\mathrm{grad}} = C_{\mathrm{grad}}(\Omega')$ where $\Omega' \subset\subset \Omega$ (that is, $\Omega'$ is open with $\overline{\Omega'} \subset \Omega$). Therefore, \eqref{ass:onSourceTermsForRegularity} being valid a.e. only in $\Omega'$ implies \eqref{ass:forRegularity} also in $\Omega'$ and the identification of $\chi = \chi_{\{\Phi=u\}}$ will hold only locally in $\Omega' \subset\subset \Omega$, yielding only local regular solutions.
\end{remark}

For uniqueness, we focus on the case where the coefficient function $a$ is constant for simplicity. The general case requires additional technical estimates and regularity which we leave to the reader. 
\begin{theorem}[\textsc{Uniqueness of regular solutions}]\label{thm:uniquenessElliptic}
Let $a' \equiv 0$, $h_1, h_2 \in L^\infty(\Omega)$, \eqref{ass:onConstants}, and suppose that 
\begin{align}
&\gamma_0:=\min\left(c_1-(b_2- b_1)^+, c_2 -(b_1-b_2)^+\right) > 0\label{ass:forUniqueness0}
\end{align}
and
\begin{align}
&f>ag+ a\alpha(M-m)\left(2 + \gamma_0^{-1}(b_1+b_2)\right)\quad \text{a.e. in $\Omega$}.\label{ass:strongNonDegForUniqueness0}
\end{align}
Then the (regular) solution of \eqref{eq:equation} is unique.\end{theorem}
It is clear that this `strong' non-degeneracy condition \eqref{ass:strongNonDegForUniqueness0} implies, under the setting considered, the condition \eqref{ass:onSourceTermsForRegularity} and hence also the non-degeneracy condition \eqref{ass:forRegularity}.

Let us now begin the procedure for proving Theorem \ref{thm:existence}. It becomes convenient to start with a comprehensive study of an auxiliary problem that generalises the system for the $\theta_i$ in \eqref{eq:equation} or \eqref{eq:weakFormWeakSoln} since it will turn out that many properties of the system can be derived independently (in some sense) of the precise nature of the characteristic function appearing on the right-hand sides of \eqref{eq:equation} and \eqref{eq:weakFormWeakSoln}; boundedness and non-negativity of the function are sufficient to conclude many (but not at all) properties of the solutions.
\subsection{Study of a weakly-coupled auxiliary problem}\label{sec:auxProblem}
We consider the following auxiliary problem for given $\sigma \in L^\infty(\Omega)$ with $\sigma \geq 0$:
\begin{equation}\label{eq:auxProblem}
\begin{aligned}
\text{for $i=1, 2$:}\qquad - \kappa_i \Delta \vartheta_i  + c_i\vartheta_i &= h_i + (-1)^ib_i(\vartheta_1-\vartheta_2)\sigma&&\text{in $\Omega$},\\
\partial_n \vartheta_i &= 0&&\text{on $\partial\Omega$}.
\end{aligned}
\end{equation}
We abbreviate $\norm{\sigma}{\infty}:=\norm{\sigma}{L^\infty(\Omega)}$. The function $\sigma$ here generalises the role of $\chi$ in \eqref{eq:equation1} as well as approximations of $\chi$ that are going to appear later on in the course of the paper.
\begin{prop}\label{prop:auxiliaryExistence}Let $h_1, h_2\in L^2(\Omega)$ and suppose that
\begin{equation}\label{ass:auxonConstants}
c_\sigma := \min\left(c_1 - \frac{(b_2 - b_1)^+\norm{\sigma}{\infty}}{4}, c_2 - \frac{(b_1 - b_2)^+\norm{\sigma}{\infty}}{4}\right) > 0.
\end{equation}
Then there exists a unique solution $(\vartheta_1, \vartheta_2) \in H^1(\Omega) \times H^1(\Omega) $ to \eqref{eq:auxProblem} with
\begin{align*}
\norm{\vartheta_1}{H^1(\Omega)}^2 + \norm{\vartheta_2}{H^1(\Omega)}^2 \leq  \frac{1}{\mu c_\sigma}\left(\norm{h_1}{L^2(\Omega)}^2  +  \norm{h_2}{L^2(\Omega)}^2\right)
\end{align*}
where
$\mu := \min\left(\kappa_1, \kappa_2, c_\sigma\right).$
\end{prop}
\begin{proof}
Considering the matrices
\begin{align*}
\mathbf{A} = \begin{pmatrix}
-\kappa_1\Delta &0\\
0 &-\kappa_2\Delta
\end{pmatrix}, \qquad \mathbf{B} = \begin{pmatrix}
b_1\sigma + c_1 &-b_1\sigma\\
-b_2\sigma &b_2\sigma + c_2
\end{pmatrix}, \qquad \mathbf{h} = \begin{pmatrix}
h_1\\
h_2
\end{pmatrix},
\end{align*}
the problem for $\mathbf{\vartheta} = (\vartheta_1,\vartheta_2)$ reads
\[\mathbf{A}\mathbf{\vartheta} + \mathbf{B}\mathbf{\vartheta} = \mathbf{h}.\]
The bilinear form generated by the operator on the left-hand side (taking into account the zero Neumann boundary conditions) is clearly bounded from $H^1(\Omega)\times H^1(\Omega)$ into $\mathbb{R}$. For coercivity, we begin with the calculation
\begin{align*}
\langle \mathbf{A}\mathbf{\vartheta} + \mathbf{B}\mathbf{\vartheta}, \mathbf{\vartheta} \rangle &= \kappa_1\norm{\grad \vartheta_1}{L^2(\Omega)}^2 + c_1 \norm{\vartheta_1}{L^2(\Omega)}^2  +\kappa_2\norm{\grad \vartheta_2}{L^2(\Omega)}^2 + c_2 \norm{\vartheta_2}{L^2(\Omega)}^2 + b_1\int_\Omega (\vartheta_1-\vartheta_2)\sigma\vartheta_1\\
&\quad- b_2\int_\Omega (\vartheta_1-\vartheta_2)\sigma\vartheta_2.
\end{align*}
The case $b_1=b_2$ being trivial, we suppose then that $b_1 > b_2$, so there exists $\delta>0$ such that $b_1 = b_2 + \delta$. We see that the last two terms on the right-hand side above are
\begin{align}
\nonumber \int_\Omega \sigma(\vartheta_1-\vartheta_2)(b_1\vartheta_1 - b_2\vartheta_2)&\geq b_2\int_\Omega \sigma(\vartheta_1-\vartheta_2)^2 + \delta\int_\Omega \sigma\vartheta_1(\vartheta_1-\vartheta_2)\\
\nonumber &\geq \delta\int_\Omega \sigma(\vartheta_1^2 -\vartheta_1\vartheta_2)\\
\nonumber &\geq -\frac{\delta}{4}\int_\Omega \sigma\vartheta_2^2\tag{using $\vartheta_1\vartheta_2 \leq \vartheta_1^2 + (1\slash 4)\vartheta_2^2$}\\
& \geq -\frac{\delta\norm{\sigma}{\infty}}{4}\int_\Omega \vartheta_2^2.\label{eq:usefulCalc}
\end{align}
Plugging this in above, we find, using the assumption \eqref{ass:auxonConstants}, that
\begin{align*}
\langle \mathbf{A}\mathbf{\vartheta} + \mathbf{B}\mathbf{\vartheta}, \mathbf{\vartheta} \rangle &\geq \kappa_1\norm{\grad \vartheta_1}{L^2(\Omega)}^2 + c_1 \norm{\vartheta_1}{L^2(\Omega)}^2 +\kappa_2\norm{\grad \vartheta_2}{L^2(\Omega)}^2 + \left(c_2-\frac{(b_1-b_2)\norm{\sigma}{\infty}}{4}\right) \norm{\vartheta_2}{L^2(\Omega)}^2\\
&\geq \mu\left(\norm{\vartheta_1}{H^1(\Omega)}^2 +\norm{\vartheta_2}{H^1(\Omega)}^2\right)\\
&= \mu\norm{\mathbf \vartheta}{H^1(\Omega) \times H^1(\Omega)}^2,
\end{align*}
i.e., the operator $\mathbf{A} + \mathbf{B}$ is coercive. Hence, by the Riesz theorem, there exists a unique solution $(\vartheta_1, \vartheta_2) \in H^1(\Omega) \times H^1(\Omega)$ to the problem for every $(h_1, h_2) \in L^2(\Omega)\times L^2(\Omega)$.  

Regarding the \emph{a priori} estimate, we see that  
\begin{align*}
\int_\Omega h_1\vartheta_1 + h_2\vartheta_2 &\leq \frac{1}{4\rho}\norm{h_1}{L^2(\Omega)}^2 + \rho\norm{\vartheta_1}{L^2(\Omega)}^2 + \frac{1}{4\gamma}\norm{h_2}{L^2(\Omega)}^2 + \gamma\norm{\vartheta_2}{L^2(\Omega)}^2\\
&\leq \frac{1}{2c_1}\norm{h_1}{L^2(\Omega)}^2 + \frac{c_1}{2}\norm{\vartheta_1}{L^2(\Omega)}^2 + \frac{2}{4c_2-(b_1-b_2)\norm{\sigma}{\infty}}\norm{h_2}{L^2(\Omega)}^2\\
&\quad + \frac{1}{2}\left(c_2-\frac{(b_1-b_2)\norm{\sigma}{\infty}}{4}\right) \norm{\vartheta_2}{L^2(\Omega)}^2
\end{align*}
where we used Young's inequality with $\rho=c_1\slash 2$ and $\gamma=(1\slash 2)\left(c_2-(b_1-b_2)\norm{\sigma}{\infty}\slash 4\right)$ and the same manipulations as before in \eqref{eq:usefulCalc} to deal with the rightmost term. Finally, bounding $4c_2 - (b_1-b_2)\norm{\sigma}{\infty} \geq 4c_\sigma$ on the penultimate term and combining with the above estimate,
\begin{align*}
&\kappa_1\norm{\grad \vartheta_1}{L^2(\Omega)}^2 + \frac{c_1}{2} \norm{\vartheta_1}{L^2(\Omega)}^2 +\kappa_2\norm{\grad \vartheta_2}{L^2(\Omega)}^2 + \frac 12\left(c_2-\frac{(b_1-b_2)\norm{\sigma}{\infty}}{4}\right) \norm{\vartheta_2}{L^2(\Omega)}^2\\
&\quad \leq  \frac{1}{2c_1}\norm{h_1}{L^2(\Omega)}^2  + \frac{1}{2c_\sigma}\norm{h_2}{L^2(\Omega)}^2.
\end{align*}
The argument is analogous in the case $b_2 > b_1$.
\end{proof}
Given existence, let us prove some results on the boundedness of the solutions. For this purpose, recall $m$ and $M$ from \eqref{eq:mAndM}. 
\begin{prop}\label{prop:LInftyEstimates}Let $h_1, h_2\in L^2(\Omega)$ and \eqref{ass:auxonConstants} hold. Then
\[\vartheta_1, \vartheta_2 \geq m\]
and
\[\vartheta_1, \vartheta_2 \leq  M.\]
\end{prop}
\begin{proof}
Testing the $\vartheta_i$ equation with $(\vartheta_i-M)^+$ for a constant $M>0$ to be fixed, we find after writing $c_i\vartheta_i = c_i(\vartheta_i -M) + c_iM$,
\begin{align*}
\int_\Omega \kappa_i |\grad (\vartheta_i-M)^+|^2 + c_i|(\vartheta_i-M)^+|^2 &= \int_\Omega (h_i-c_iM)(\vartheta_i-M)^+ + (-1)^ib_i(\vartheta_1-\vartheta_2)\sigma(\vartheta_i-M)^+.
\end{align*}
Adding the two equations for $i=1,2$, we see that
\begin{align*}
&\int_\Omega \kappa_1 |\grad (\vartheta_1-M)^+|^2 + \kappa_2 |\grad (\vartheta_2-M)^+|^2 + c_1|(\vartheta_1-M)^+|^2  + c_2|(\vartheta_2-M)^+|^2\\
 &= \int_\Omega (h_1-c_1M)(\vartheta_1-M)^+ +  (h_2-c_2M)(\vartheta_2-M)^+ + b_2(\vartheta_1-\vartheta_2)\sigma(\vartheta_2-M)^+ - b_1(\vartheta_1-\vartheta_2)\sigma(\vartheta_1-M)^+.
\end{align*}
Assume for now that $b_1 > b_2$. Defining $v_i = \vartheta_1-M$, we manipulate the final two terms on the right-hand side by adding and subtracting $M$ as follows:
\begin{align*}
b_2(\vartheta_1-\vartheta_2)(\vartheta_2-M)^+ - b_1(\vartheta_1-\vartheta_2)(\vartheta_1-M)^+   &= b_2(v_1-v_2) v_2^+ - b_1(v_1-v_2) v_1^+  \\
&\leq \left(b_2v_1^+v_2^+ -b_2|v_2^+|^2 - b_1|v_1^+|^2 +b_2 v_2^+v_1^+\right)\\
&= \left(v_1^+ - v_2^+)(b_2v_2^+ - b_1v_1^+\right)\\
&\leq \frac{b_1-b_2}{4}|v_2^+|^2\tag{as in \eqref{eq:usefulCalc}}\\
&= \frac{b_1-b_2}{4}|(\vartheta_2-M)^+|^2.
\end{align*}
If we now choose $M$ as given in \eqref{eq:mAndM}, then we will obtain $\vartheta_i \leq M$ for $i=1,2$. A similar argument holds in the case when $b_1 < b_2$ and we conclude the result. For the lower bound, we instead test with $(\vartheta_i-m)^-$  and perform the same manipulations.
\end{proof}

A consequence of the above result is that if $h_1, h_2 \geq 0$ and \eqref{ass:auxonConstants} holds, then $\vartheta_1, \vartheta_2 \geq 0$, whereas if $h_1, h_2 \in L^\infty(\Omega)$ and \eqref{ass:auxonConstants} holds, then 
\begin{equation}
 \norm{\vartheta_1-\vartheta_2}{L^\infty(\Omega)} \leq M-m.\label{eq:auxLinftyEstimateTempDiff}
\end{equation}
\begin{prop}[\textsc{Comparison principle}]\label{prop:comparisonPrinciple}
Let $h_1, h_2\in L^2(\Omega)$ and let \eqref{ass:onConstantsForRegularity} and \eqref{ass:auxonConstants} hold. Then $\vartheta_1 \geq \vartheta_2 \geq 0.$
\end{prop}
\begin{proof}
Dividing each $\vartheta_i$ equation by the diffusion coefficient, the difference satisfies
\[-\Delta (\vartheta_2-\vartheta_1) + \frac{c_2}{\kappa_2}\vartheta_2 - \frac{c_1}{\kappa_1}\vartheta_1 = \frac{h_2}{\kappa_2}-\frac{h_1}{\kappa_1} -\left(\frac{b_2}{\kappa_2}+\frac{b_1}{\kappa_1}\right)(\vartheta_2-\vartheta_1)\sigma,\]
whence writing $(c_2\slash \kappa_2)\vartheta_2 = (c_2\slash \kappa_2)(\vartheta_2-\vartheta_1) + (c_2\slash \kappa_2)\vartheta_1$ and testing with $(\vartheta_2-\vartheta_1)^+$, we get
\begin{align*}
\norm{\grad (\vartheta_2-\vartheta_1)^+}{L^2(\Omega)}^2 + \frac{c_2}{\kappa_2}\norm{(\vartheta_2-\vartheta_1)^+}{L^2(\Omega)}^2 + \left(\frac{c_2}{\kappa_2}-\frac{c_1}{\kappa_1}\right)\int_\Omega \vartheta_1(\vartheta_2-\vartheta_1)^+ + \left(\frac{b_1}{\kappa_1}+\frac{b_2}{\kappa_2}\right)\int_\Omega |(\vartheta_2-\vartheta_1)^+|^2\sigma
\leq 0,
\end{align*}
giving $\vartheta_2 \leq \vartheta_1$ (we can neglect the third term on the left-hand side above thanks to the non-negativity of $\vartheta_1$ assured by \eqref{ass:onConstantsForRegularity}).
\end{proof}

\begin{lem}\label{lem:LinftyBoundTheta1}
Let $h_1 \in L^\infty(\Omega)$, $h_2\in L^2(\Omega)$ and let \eqref{ass:onConstantsForRegularity} and \eqref{ass:auxonConstants} hold. Then 
\[\norm{\vartheta_1}{L^\infty(\Omega)} \leq \frac{\norm{h_1}{L^\infty(\Omega)}}{c_1}.\]
\end{lem}
\begin{proof}
Taking $v=\vartheta_1-K$ for a constant $K>0$ to be fixed, we find that $v$ satisfies the equation
\[-\kappa_1\Delta v + c_1v + c_1K + b_1(\vartheta_1-\vartheta_2)\sigma = h_1,\]
whence testing with $v^+$:
\begin{align*}
\int_\Omega \kappa_1|\grad v^+|^2 + c_1|v^+|^2 + b_1(\vartheta_1-\vartheta_2)\sigma v^+ = \int_\Omega (h_1-c_1K)v^+.
\end{align*}
Realising that $\vartheta_1 \geq  \vartheta_2$ under the assumptions (see Proposition \ref{prop:comparisonPrinciple}), if we choose $K:=\norm{h_1}{L^\infty(\Omega)}\slash c_1$, the right-hand side is non-positive, giving $v \leq 0$.
\end{proof}
It will be useful to obtain an estimate on the continuous dependence of the solutions on the function $\sigma$. We will use this later in \S \ref{sec:uniquenessElliptic} to prove uniqueness of regular solutions by assuming that
\begin{equation}\label{eq:defnOfGammai}
\gamma_1 := (c_1-(b_2-b_1)^+\norm{\hat\sigma}{\infty}) > 0\quad\text{and}\quad \gamma_2 :=(c_2-(b_1-b_2)^+\norm{\hat\sigma}{\infty}) > 0.
\end{equation}
\begin{prop}[\textsc{$L^1$-continuous dependence}]\label{prop:L1ctsDepElliptic}Let \eqref{ass:auxonConstants} and \eqref{eq:defnOfGammai} hold. Define $\hat \vartheta_i$ as the solution of \eqref{eq:auxProblem} corresponding to data $\hat h_i,$ $\hat \sigma$ under the same assumptions and suppose additionally that $h_i, \hat h_i \in L^\infty(\Omega)$ for $i=1,2$.
Then 
\begin{align}
\gamma_1\norm{\vartheta_1-\hat \vartheta_1}{L^1(\Omega)} +
\gamma_2\norm{\vartheta_2-\hat \vartheta_2}{L^1(\Omega)} &\leq \norm{h_1-\hat h_1}{L^1(\Omega)} + \norm{h_2-\hat h_2}{L^1(\Omega)}  + (M-m)(b_1+b_2)\norm{\sigma-\hat\sigma}{L^1(\Omega)}\label{eq:L1CtsDependence}.
\end{align}
\end{prop}
\begin{proof}
The difference $\vartheta_i-\hat\vartheta_i$ satisfies 
\begin{align*}
- \kappa_i \Delta (\vartheta_i -\hat \vartheta_i)+ c_i(\vartheta_i-\hat \vartheta_i) &= h_i - \hat h_i + (-1)^ib_i((\vartheta_1-\vartheta_2)\sigma-(\hat \vartheta_1-\hat \vartheta_2)\hat \sigma) \\
 &= h_i - \hat h_i + (-1)^ib_i((\vartheta_1-\vartheta_2)(\sigma-\hat \sigma) + (\vartheta_1-\hat\vartheta_1 + \hat \vartheta_2 -\vartheta_2 )\hat \sigma).
\end{align*}
The idea is to test with $\mathrm{sign}(\vartheta_i-\hat \vartheta_i)$. To do this rigorously, we define the usual truncation function at a height $\epsilon$:
\[T_\epsilon(s) = \begin{cases}
\epsilon &: s > \epsilon\\
s &: |s| \leq \epsilon\\
-\epsilon &: s < -\epsilon,
\end{cases}
\]
and test the equation with $\epsilon^{-1}T_\epsilon(\vartheta_i - \hat\vartheta_i)$. The gradient term can be neglected since for all $v \in H^1(\Omega)$, we have $\grad v  \cdot \grad T_\epsilon(v) = T'_\epsilon(v)|\grad v|^2 \geq 0$, and furthermore, observing that  $\epsilon^{-1}vT_\epsilon(v) \to |v|$ a.e.,
we obtain, as $\epsilon \to 0$,
\begin{align*}
\frac{1}{\epsilon}\int_\Omega vT_\epsilon(v) 
\to \norm{v}{L^1(\Omega)}.
\end{align*}
Bearing in mind the absolute bound $|\epsilon^{-1}T_\epsilon(s)| \leq 1$ and the $L^\infty$ estimate \eqref{eq:auxLinftyEstimateTempDiff}, we obtain after sending $\epsilon \to 0$ the estimates
\begin{align*}
c_1\norm{\vartheta_1-\hat \vartheta_1}{L^1(\Omega)}  + b_1\int_{\Omega} |\vartheta_1-\hat \vartheta_1|\hat \sigma &\leq \norm{h_1-\hat h_1}{L^1(\Omega)} + (M-m)b_1\norm{\sigma-\hat\sigma}{L^1(\Omega)} + b_1\int_{\Omega}\hat \sigma|\hat \vartheta_2-\vartheta_2|,\\
c_2\norm{\vartheta_2-\hat \vartheta_2}{L^1(\Omega)}  + b_2\int_{\Omega} |\vartheta_2-\hat \vartheta_2|\hat \sigma
&\leq \norm{h_2-\hat h_2}{L^1(\Omega)} + (M-m)b_2\norm{\sigma-\hat\sigma}{L^1(\Omega)} + b_2\int_{\Omega} \hat \sigma |\hat \vartheta_1-\vartheta_1|.
\end{align*}
Adding the above inequalities leads to
\begin{align*}
c_1\norm{\vartheta_1-\hat \vartheta_1}{L^1(\Omega)}  + c_2\norm{\vartheta_2-\hat \vartheta_2}{L^1(\Omega)}  &\leq \norm{h_1-\hat h_1}{L^1(\Omega)} + \norm{h_2-\hat h_2}{L^1(\Omega)} + (M-m)(b_1+b_2)\norm{\sigma-\hat\sigma}{L^1(\Omega)} \\
&\quad + (b_1-b_2)\int_{\Omega}\hat \sigma|\hat \vartheta_2-\vartheta_2| + (b_2-b_1)\int_{\Omega} \hat \sigma |\hat \vartheta_1-\vartheta_1|.
\end{align*}
Suppose that $b_1 > b_2$.  Then the final term is non-positive and can be neglected, and we use the boundedness of $\hat \sigma$ on the penultimate term and we get
\begin{align*}
c_1\norm{\vartheta_1-\hat \vartheta_1}{L^1(\Omega)}  + (c_2-(b_1-b_2)\norm{\hat \sigma}{\infty})\norm{\vartheta_2-\hat \vartheta_2}{L^1(\Omega)} &\leq \norm{h_1-\hat h_1}{L^1(\Omega)} + \norm{h_2-\hat h_2}{L^1(\Omega)}\\
&\quad +  (M-m)(b_1+b_2)\norm{\sigma-\hat\sigma}{L^1(\Omega)}.
\end{align*}
Similarly, if $b_1 < b_2$,
\begin{align*}
(c_1-(b_2-b_1)\norm{\hat \sigma}{\infty})\norm{\vartheta_1-\hat \vartheta_1}{L^1(\Omega)}  + c_2\norm{\vartheta_2-\hat \vartheta_2}{L^1(\Omega)} &\leq \norm{h_1-\hat h_1}{L^1(\Omega)} + \norm{h_2-\hat h_2}{L^1(\Omega)} \\
&\quad + (M-m)(b_1+b_2)\norm{\sigma-\hat\sigma}{L^1(\Omega)},
\end{align*}
so that, combining both cases, one has \eqref{eq:L1CtsDependence}.
\end{proof}
\subsection{Regularisation of the problem}\label{sec:regFunctions}
Returning to the elliptic problem under study, we make smooth the characteristic function appearing in \eqref{eq:equation}. For $\epsilon>0$, let $\chi_\epsilon \in C^{0,1}(\mathbb{R})$ be a family of functions parameterised by $\epsilon$ satisfying the following properties:
\begin{enumerate}[(i)]
\item $\chi_\epsilon(s) = 1$ for $s \leq 0$,
\vspace{-0.2cm}\item $\lim_{s \to \infty}\chi_\epsilon(s)=0$,
\vspace{-0.2cm}\item $\chi_\epsilon' \leq 0$,
\vspace{-0.2cm}\item there exists a constant $C_\nu > 0$ such that
\begin{equation*}
\chi_\epsilon(s)s \leq \epsilon C_\nu\quad\text{ for $s>0$,}
\end{equation*}
\vspace{-0.6cm}\item $\chi_\epsilon(s) \to 1-H(s)$ pointwise for $s \neq 0$.
\end{enumerate}
We look for existence of solutions to the following approximation of \eqref{eq:equation}: find $(\theta_1^\epsilon, \theta_2^\epsilon,u^\epsilon, \Phi^\epsilon)$ such that
\begin{equation}\label{eq:regularised}
\begin{alignedat}{3}
\text{for $i=1, 2$:}\qquad \int_\Omega \kappa_i \grad \theta_i^\epsilon \cdot \grad \eta + c_i\theta_i^\epsilon\eta &= \int_\Omega (h_i + (-1)^ib_i(\theta_1^\epsilon-\theta_2^\epsilon)\chi_{\epsilon}(\Phi^\epsilon-u^\epsilon))\eta  \quad &&\forall \eta \in H^1(\Omega),\\
\int_\Omega \grad \Phi^\epsilon \cdot \grad \xi &= \int_\Omega (\alpha(\theta_1^\epsilon - \theta_2^\epsilon)\chi_{\epsilon}(\Phi^\epsilon-u^\epsilon) + g)\xi  &&\forall \xi \in H^1_0(\Omega),\\
u^\epsilon \in \mathbb{K}(\Phi^\epsilon) : \int_\Omega a(\theta_1^\epsilon)\grad u^\epsilon \cdot \grad (u^\epsilon- v) &\leq \int_\Omega f(u^\epsilon- v) &&\forall v \in \mathbb{K}(\Phi^\epsilon).
\end{alignedat}
\end{equation}
\begin{remark}
An alternative approach is to penalise the quasi-variational inequality by using a Moreau--Yosida regularisation of the constraint (more details can be found for instance in \cite{MR3903798}) or to use a bounded penalisation as, for example, in \cite[\S 5:3]{Rodrigues}. 
\end{remark}
This approximating problem still poses difficulties due to the nontrivial coupling and the non-linearities appearing in the equations. We ``linearise" \eqref{eq:regularised} and modify it for an argument amenable to a fixed point theorem.
Consider for given $\phi, w \in L^2(\Omega)$ the system
\begin{subequations}\label{eq:fpSystem}
\begin{align}
\text{for $i=1, 2$:}\qquad \int_\Omega \kappa_i \grad \theta_i^\epsilon \cdot \grad \eta + c_i\theta_i^\epsilon\eta &= \int_\Omega (h_i + (-1)^ib_i(\theta_1^\epsilon-\theta_2^\epsilon)\chi_{\epsilon}(\phi-w))\eta &&\forall \eta \in H^1(\Omega),\\
\int_\Omega \grad  \Phi^\epsilon \cdot \grad \xi &= \int_\Omega (\alpha(\theta_1^\epsilon - \theta_2^\epsilon)\chi_{\epsilon}(\phi-w) + g)\xi &&\forall \xi \in H^1_0(\Omega),\label{eq:fpPhi}\\
u^\epsilon \in \mathbb{K}(\Phi^\epsilon) : \int_\Omega a(\theta_1^\epsilon)\grad u^\epsilon \cdot \grad(u^\epsilon- v) &\leq \int_\Omega f(u^\epsilon- v)  &&\forall v \in \mathbb{K}(\Phi^\epsilon).\label{eq:fpU}
\end{align}
\end{subequations}
This is a completely uncoupled system: we can first solve the system for $\theta_i^\epsilon$ and use $\theta_i^\epsilon$ as data to solve for $\Phi^\epsilon$ and then $u^\epsilon$.

\begin{lem}\label{prop:existenceForAllButu}Let $f, g, h_1, h_2 \in L^2(\Omega)$, and suppose that \eqref{ass:onConstants} holds.  Then there exists a unique solution $$(\theta_1^\epsilon, \theta_2^\epsilon, \Phi^\epsilon, u^\epsilon) \in H^1(\Omega) \times H^1(\Omega)  \times H_0^1(\Omega)\times H_0^1(\Omega) $$ to \eqref{eq:fpSystem} with
\begin{align*}
\norm{\theta_i^\epsilon}{H^1(\Omega)} + \norm{\Phi^\epsilon}{H^1_0(\Omega)} + \norm{u^\epsilon}{H^1_0(\Omega)}  &\leq C \quad\text{independent of $\phi, w$ and $\epsilon$.}
\end{align*}
\end{lem}
\begin{proof}
The statements for $\theta_i^\epsilon$ are a simple consequence of Proposition \ref{prop:auxiliaryExistence}; we just need to choose $\sigma:= \chi_\epsilon(\phi-w)$ there. Given $\theta_i^\epsilon$, the unique existence for $\Phi^\epsilon \in H_0^1(\Omega)$ is immediate since the source term in \eqref{eq:fpPhi} can be considered as given data in $L^2(\Omega)$. For the bound, take $\Phi$ as the test function in \eqref{eq:fpPhi} and use Poincar\'e's inequality to  obtain
\begin{align*}
\int_\Omega |\grad \Phi^\epsilon|^2 &\leq \alpha\norm{\theta_1^\epsilon-\theta_2^\epsilon}{L^2(\Omega)}\norm{\Phi^\epsilon}{L^2(\Omega)} + \norm{g}{L^2(\Omega)}\norm{\Phi^\epsilon}{L^2(\Omega)} \leq C_P\left(\alpha\norm{\theta_1^\epsilon-\theta_2^\epsilon}{L^2(\Omega)}+C\norm{g}{L^2(\Omega)}\right)\norm{\grad \Phi^\epsilon}{L^2(\Omega)}.
\end{align*}
Regarding the variational inequality, observe that the associated elliptic operator is monotone:
 \begin{align*}
\int_\Omega a(\theta_1)\grad u \cdot \grad(u-v) -\int_\Omega a(\theta_1)\grad v \cdot \grad(u-v)
&= \int_\Omega a(\theta_1)|\grad u - \grad v|^2\geq \lambda_1 \int_\Omega |\grad u - \grad v|^2,
\end{align*}
as well as bounded. Hence, existence and uniqueness follows by Lions--Stampacchia e.g. \cite[Theorem 3.1, \S 4:3]{Rodrigues}.
Testing the inequality with $\Phi^\epsilon$, using monotonicity along with the bounds on $a$, we obtain
\begin{align*}
\lambda_1\norm{\grad u^\epsilon}{L^2(\Omega)}^2 &\leq  \lambda_2\norm{\grad u^\epsilon}{L^2(\Omega)}\norm{\grad \Phi^\epsilon}{L^2(\Omega)} + \norm{f}{L^2(\Omega)}\norm{u^\epsilon}{L^2(\Omega)} +  \norm{f}{L^2(\Omega)}\norm{\Phi^\epsilon}{L^2(\Omega)},
\end{align*}
whence using Poincar\'e's inequality, Young's inequality (with epsilon) and the bound on $\Phi^\epsilon$ from before leads to a uniform estimate.

\end{proof}


Having shown that solutions of \eqref{eq:fpSystem}  exist and derived suitable bounds, we proceed with the fixed point argument.  Set $\mathbb{L}^2 := L^2(\Omega) \times L^2(\Omega)$ and  $\mathbb{H}^1_0 := H^1_0(\Omega) \times H^1_0(\Omega)$ and define the maps
\begin{equation*}
\begin{aligned}
&\mathcal Q\colon (w,\phi) \mapsto (\theta_1^\epsilon, \Phi^\epsilon)\qquad\text{and}\qquad \mathcal P\colon (\theta_1^\epsilon, \Phi^\epsilon) \mapsto (u^\epsilon,\Phi^\epsilon)
\end{aligned}
\end{equation*}
as the solution maps given through Lemma \ref{prop:existenceForAllButu}. We have shown that in fact
\begin{equation*}
\begin{aligned}
&\mathcal Q\colon \mathbb{L}^2  \to {H}^1(\Omega) \times H^1_0(\Omega) \qquad \text{and} \qquad \mathcal P\colon L^2(\Omega) \times H^1_0(\Omega) \to \mathbb{H}^1_0.
\end{aligned}
\end{equation*}
We consider the composition map
\begin{align*}
\mathcal S\colon \mathbb{L}^2 \to \mathbb{L}^2, \qquad (w,\phi)  \overset{\mathcal Q}{\mapsto} (\theta_1^\epsilon, \Phi^\epsilon) \overset{\mathcal P}{\mapsto} (u^\epsilon, \Phi^\epsilon),
\end{align*}
i.e., $\mathcal S= \mathcal P \circ \mathcal Q$ and we will show that $\mathcal S$ has a fixed point, giving existence for the regularised problem \eqref{eq:regularised}.  

To this end, let $C^*$ be a constant that exceeds twice the largest constants from the \emph{a priori} estimates on $u^\epsilon$ and $\Phi^\epsilon$ in Lemma \ref{prop:existenceForAllButu} and define the set
\[\mathbb D := \{ (v,\psi)  \in \mathbb{L}^2: \norm{v}{L^2(\Omega)}+\norm{\psi}{L^2(\Omega)}  \leq C^*\}.\]
Then $\mathcal{S}\colon \mathbb D \to \mathbb D$. The next theorem shows that $\mathcal{S}$ has fixed points and is the main result of this section.
\begin{prop}[\textsc{Existence for the regularised problem}]\label{prop:existenceForRegularisedProblem}Let $f, g, h_1, h_2 \in L^2(\Omega)$ and suppose that \eqref{ass:onConstants} holds.  Then the system \eqref{eq:regularised} has a solution
\[(\theta_1^\epsilon, \theta_2^\epsilon, u^\epsilon, \Phi^\epsilon) \in H^1(\Omega)^2 \times H_0^1(\Omega)^2\] with
\[\norm{\theta_1^\epsilon}{H^1(\Omega)} + \norm{\theta_2^\epsilon}{H^1(\Omega)} + \norm{u^\epsilon}{H^1_0(\Omega)} + \norm{\Phi^\epsilon}{H^1_0(\Omega)} \leq C \quad\text{uniformly in $\epsilon$}.\]
\end{prop}
\begin{proof}	
Let us prove continuity of $\mathcal S$. Take $(w^n, \phi^n) \to (w,\phi)$ in $\mathbb{L}^2$. Denote the solution associated to the data $(w^n,\phi^n)$ as $(\theta_1^n, \theta_2^n, u^n, \Phi^n)$ so that
\begin{equation}\label{eq:regularisedFPcts}
\begin{aligned}
\text{for $i=1, 2$:}\qquad \int_\Omega  \kappa_i \grad \theta_i^n\cdot \grad \eta+ c_i\theta_i^n\eta&= \int_\Omega (h_i + (-1)^ib_i(\theta_1^n-\theta_2^n)\chi_{\epsilon}(\phi^n-w^n))\eta&&\forall \eta \in H^1(\Omega),\\
\int_\Omega \grad \Phi^n\cdot \grad \xi &= \int_\Omega (\alpha(\theta_1^n- \theta_2^n)\chi_{\epsilon}(\phi^n-w^n) + g)\xi&&\forall \xi \in H^1_0(\Omega),\\
u^n \in \mathbb{K}(\Phi^n) : \int_\Omega a(\theta_1^n)\grad u^n \cdot \grad(u^n-v) &\leq \int_\Omega f(u^n - v)&&\forall v \in \mathbb{K}(\Phi^n).
\end{aligned}
\end{equation}
By the bound in Lemma \ref{prop:existenceForAllButu}, we obtain the existence of $\theta_i$ and $\Phi$ such that $\theta_i^n \weaklyto \theta_i$ in $H^1(\Omega)$ and $\Phi^n \weaklyto \Phi$ in $H^1_0(\Omega)$ for a subsequence that we have relabelled. Hence, for a further subsequence, $(\theta_1^n-\theta_2^n)\chi_\epsilon(\phi^n-w^n)$ converges pointwise a.e. to $(\theta_1-\theta_2)\chi_\epsilon(\phi-w)$ and by Lebesgue's dominated convergence theorem, also in $L^2(\Omega)$.
This allows us to show that the limits satisfy
\begin{align*}
\text{for $i=1, 2$:}\qquad \int_\Omega \kappa_i \grad \theta_i \cdot \grad \eta + c_i\theta_i\eta &= \int_\Omega (h_i + (-1)^ib_i(\theta_1-\theta_2)\chi_{\epsilon}(\phi-w))\eta&&\forall \eta \in H^1(\Omega),\\
\int_\Omega \grad \Phi \cdot \grad \xi &= \int_\Omega (\alpha(\theta_1 - \theta_2)\chi_{\epsilon}(\phi-w) + g)\xi &&\forall \xi \in H^1_0(\Omega).
\end{align*}
We also find that $-\Delta\Phi^n \to -\Delta \Phi$ in $L^2(\Omega)$ because the right-hand side of the equation for $\Phi^n$ converges in $L^2(\Omega)$, giving (not using elliptic regularity but merely coercivity of the Laplacian as an operator from $H^1_0(\Omega)$ into $H^{-1}(\Omega)$) the strong convergence $\Phi^n \to \Phi$ in $H^1_0(\Omega)$ and likewise $\theta_i^n \to \theta_i$ in $H^1(\Omega)$. Since the solution to the above system is uniquely determined for fixed $\phi$ and $w$, it follows that the convergences stated above hold for the entire sequences, which shows that $\mathcal Q\colon \mathbb{L}^2 \to H^1(\Omega)\times H^1_0(\Omega)$ is continuous.

From Lemma \ref{prop:existenceForAllButu}, we also have the existence of $u \in H^1_0(\Omega)$ such that $u^n \weaklyto u$ in $H^1_0(\Omega)$ for a subsequence that we have again relabelled. The strong convergence in $H^1_0(\Omega)$ of the obstacles $\Phi^n$ implies Mosco convergence for the constraint sets: indeed, given a limiting function $v \in H^1_0(\Omega)$ with $v \leq \Phi$, the sequence $\{v^n\}$ defined by $v^n := v - \Phi + \Phi^n$  satisfies $v^n \leq \Phi^n$ and converges to $v$ strongly in $H^1_0(\Omega)$. Testing the variational inequality for $u^n$ with such a $v^n$, using Minty's lemma \cite[Lemma 1.5, \S III]{MR1786735}, it suffices to pass to the limit in
\begin{align}\label{eq:mintyVI}
u^n \in \mathbb{K}(\Phi^n) : \int_\Omega a(\theta_1^n)\grad v^n  \cdot \grad (u^n-v^n) &\leq \int_\Omega f(u^n - v^n)\quad \forall v^n \in \mathbb{K}(\Phi^n).
\end{align}
instead of the variational inequality in \eqref{eq:regularisedFPcts}. Thanks to the strong convergence of $v^n$, the fact that $\theta_1^n \to \theta_1$ in $L^2(\Omega)$ and the continuity of the function $a$, we obtain for a subsequence that
\[a(\theta_1^n)\grad v^n \to a(\theta_1)\grad v \quad\text{pointwise a.e.}\]
and 
\[|a(\theta_1^n)\grad v^n| \leq \lambda_2|\grad v^n|\quad\text{pointwise a.e.}\]
Since the right-hand side above converges pointwise a.e. and in $L^2(\Omega)$ to $|\grad v|$, by the generalised Lebesgue's dominated convergence theorem, we get
\[a(\theta_1^n)\grad v^n \to a(\theta_1)\grad v \quad\text{in $L^2(\Omega)$}.\]
This lets us pass to the limit in \eqref{eq:mintyVI} and we end up with, after using Minty's lemma again to return to the original form,
\begin{align*}
u\in \mathbb{K}(\Phi) : \int_\Omega a(\theta_1)\grad u \cdot \grad(u-v) &\leq \int_\Omega f(u - v) \quad \forall v \in \mathbb{K}(\Phi).
\end{align*}
Once again, we see that there is no need to pass to subsequences since the variational inequality above has a unique solution for a given obstacle (and $\Phi$ has already been uniquely determined before). This shows that $\mathcal P\colon L^2(\Omega) \times H^1_0(\Omega) \to \mathbb{H}^1_0$ is strong-weak continuous and hence continuous into $\mathbb{L}^2$, giving continuity of $\mathcal S \colon \mathbb{L}^2 \to \mathbb{L}^2$.

To see that $\mathcal S\colon \mathbb D \to \mathbb D$ is compact, it suffices to prove that for any sequence $(w^n, \phi^n) \in \mathbb D$, we can find a subsequence $n_j$ with $\mathcal  S(w^{n_j}, \phi^{n_j})$ convergent. It follows by Lemma \ref{prop:existenceForAllButu} that $\mathcal S(w^n, \phi^n) =: (u^n, \Phi^n)$ is bounded uniformly in $\mathbb{H}^1_0$, and hence by the compact embedding into $\mathbb{L}^2$, $\mathcal S(w^n, \phi^n)$ has a convergent subsequence in $\mathbb{L}^2$.

Finally, an application of Schauder's fixed point theorem provides the result.
\end{proof}

\subsection{Passage to the limit in the regularisation parameter}\label{sec:passageLimit}
Having shown the existence of a solution $(\theta_1^\epsilon, \theta_2^\epsilon,u^\epsilon, \Phi^\epsilon)$ to the problem \eqref{eq:regularised}, 
from the estimates provided in Proposition \ref{prop:existenceForRegularisedProblem}, we will now send $\epsilon \to 0$. We obtain the existence of $(\theta_1, \theta_2, \Phi, u)$ and $\chi$ such that (for subsequences that we have relabelled):
\begin{equation}\label{eq:h1converences}
\begin{aligned}
\theta_i^\epsilon &\weaklyto \theta_i \qquad&&\text{in $H^1(\Omega)$,} \\
\Phi^\epsilon &\weaklyto \Phi &&\text{in $H^1_0(\Omega)$,}\\
u^\epsilon &\weaklyto u &&\text{in $H^1_0(\Omega)$,}\\
\chi_\epsilon(\Phi^\epsilon-u^\epsilon) &\weaklystar \chi &&\text{in $L^\infty(\Omega)$}.
\end{aligned}
\end{equation}
This, thanks to $H^1(\Omega) \ctsCompact L^2(\Omega)$, implies
\[(\theta_1^\epsilon-\theta_2^\epsilon)\chi_\epsilon(\Phi^\epsilon-u^\epsilon) \weaklyto (\theta_1-\theta_2)\chi \quad \text{in $L^2(\Omega)$},\]
leading, via the compact embedding $L^2(\Omega) \ctsCompact H^{-1}(\Omega)$, to $-\Delta \Phi^\epsilon \to -\Delta \Phi \text{ in $H^{-1}(\Omega)$}$ and thus
\[\Phi^\epsilon \to \Phi \quad \text{in $H^{1}_0(\Omega)$}.\]
Similarly,
\[\theta_i^\epsilon \to \theta_i \quad \text{in $H^{1}(\Omega)$}.\]
By using almost identical arguments to the proof of Proposition \ref{prop:existenceForRegularisedProblem}, it is not too difficult to pass to the limit in \eqref{eq:regularised} and doing so, we find
\begin{equation*}
\begin{alignedat}{3}
\text{for $i=1, 2$:}\qquad \int_\Omega \kappa_i \grad \theta_i \cdot \grad \eta + c_i\theta_i\eta &= \int_\Omega (h_i + (-1)^ib_i(\theta_1-\theta_2)\chi)\eta \quad &&\forall \eta \in H^1(\Omega),\\
\int_\Omega \grad \Phi\cdot \grad \xi &= \int_\Omega (\alpha(\theta_1- \theta_2)\chi + g)\xi &&\forall \xi \in H^1_0(\Omega),\\
u\in \mathbb{K}(\Phi) : \int_\Omega a(\theta_1)\grad u \cdot \grad(u-v) &\leq \int_\Omega f(u - v)  &&\forall v \in \mathbb{K}(\Phi).
\end{alignedat}
\end{equation*}
Our task now is to characterise $\chi$.
\begin{lem}\label{lem:XiHeaviside}The limit $\chi$ satisfies $\chi \leq \chi_{\{u=\Phi\}}$ a.e. and hence $\chi \in 1-H(\Phi-u)$.
\end{lem}
\begin{proof}
Suppose that $\chi_\epsilon$ (which approximates the characteristic function) is such that $\chi_\epsilon(r)=0$ for $r \geq \rho_\epsilon$ where $\rho_\epsilon \searrow 0$ as $\epsilon \to 0$ ; such a sequence $\{\rho_\epsilon\}$ must exist due to the properties we assumed for $\chi_\epsilon$ in \S \ref{sec:regFunctions}. Then we have
\begin{align*}
\int_\Omega \chi_\epsilon(\Phi^\epsilon-u^\epsilon)(\Phi^\epsilon-u^\epsilon-\rho_\epsilon)^+ =0
\end{align*}
because $\Phi^\epsilon-u^\epsilon \geq 0$. The second term of the integrand is non-zero only when $\Phi^\epsilon-u^\epsilon \geq \rho_\epsilon$, in which case, $\chi_\epsilon(\Phi^\epsilon-u^\epsilon) =0$. Taking $\epsilon \to 0$ yields
\begin{equation}\label{eq:usefulId}
\int_\Omega \chi(\Phi-u)^+ = 0,
\end{equation}
which tells us that when $\Phi -u > 0$, $\chi=0$, and on the coincidence set, we know only that $\chi \in [0,1]$, so that $\chi \leq \chi_{\{u=\Phi\}}$.
\end{proof}

\subsection{Local regularity}\label{sec:localRegularity}
In this section, we derive some necessary regularity results that are needed to identify $\chi$ as the characteristic function and hence prove existence of regular solutions. Our goal is to keep the regularity on $\Omega$ to be merely $C^{0,1}$, hence we consider only interior regularity.

By applying interior regularity results for elliptic PDEs (e.g. see \cite[Theorem 1, \S 6.3.1]{MR2597943}), $\theta_i, \Phi \in H^2_{\loc}(\Omega)$ irrespective of the boundary conditions and the smoothness of the boundary. Furthermore, they are bounded above in $H^2_{\loc}$ by the $L^2(\Omega)$ norms of the respective sources and the $H^1(\Omega)$ norms of the solutions. The argument for $u$ is more involved and we will provide it in the coming lemma. First, let us briefly discuss the Lewy--Stampacchia inequality as it is needed next. 

The Lewy--Stampacchia inequality \cite{LS} allows us to  give pointwise a.e. bounds on the Laplacian (or more general operators) of solutions of obstacle problems in terms of the forcing term and obstacle. As an illustration, consider (formally) a function $w$ solving the obstacle problem for a given obstacle $\psi$:
\[w \leq \psi : \langle -\Delta w, w-v \rangle \leq 0 \quad \forall v : v \leq \psi.\]
Then, on the one hand, we know that on $\{w < \psi\}$, the boundary value problem $-\Delta w = 0$ is solved and on the other hand, when $\{w=\psi\}$, $-\Delta w = -\Delta \psi$ and $\Delta \psi \geq 0$ because the obstacle has to bend up  at the points of contact. These properties are encapsulated by the \emph{Lewy--Stampacchia inequality}
\[0 \leq \Delta w \leq (\Delta \psi)^+.\]
When a forcing term $f$ is present, the argument must be modified and the resulting inequality reads $0 \leq \Delta w + f \leq (\Delta \psi + f)^+$. We will use such an inequality in the proof of the next result.
\begin{lem}\label{lem:locReg}
If $n\leq 3$ or $a' \equiv 0$, then $u \in H^2_{\loc}(\Omega)$. If  $h_1, h_2 \in L^p(\Omega)$ for $p > n$, then $u \in H^2_{\loc}(\Omega)$ and $\theta_1, \theta_2 \in C^{0,\gamma}(\bar\Omega) \cap C^{1,\alpha}(\Omega)$.
\end{lem}
\begin{proof}
Take $\varphi \in C_c^\infty(\Omega)$ with $\varphi \geq 0$. It is not difficult to see that $u$ also solves (see \cite[\S 5:5, p.~161]{Rodrigues})
\[u \in \mathbb{K}(\Phi) : \int_\Omega a(\theta)\grad u  \cdot \grad(\varphi(u-v)) \leq \int_\Omega f \varphi(u-v)\quad \forall v \in \mathbb{K}(\Phi).\]
This means that $u$ is a so-called \emph{local solution}, a notion introduced by Br\'ezis in \cite{MR428137}. Using this variational inequality, one finds that the function $\tilde u := \varphi u$ satisfies 
\begin{align}\label{eq:localSolnStuff}
\tilde u \in \mathbb{K}(\varphi\Phi) : \int_\Omega a(\theta_1)\grad \tilde u \cdot \grad(\tilde u-v) &\leq \int_\Omega \tilde f(\tilde u - v) \quad \forall v \in \mathbb{K}(\varphi\Phi)
\end{align}
with source term 
\begin{align*}
\tilde f &= \varphi f + uA_\theta \varphi - 2a(\theta_1)\grad u \grad \varphi,
\end{align*}
where we recall the definition $A_\theta  = -\grad\cdot(a(\theta_1)\grad (\cdot))$. Since $\varphi$ has compact support, suppose that $\varphi \equiv 0$ on $\Omega \setminus \Omega'$ where $\Omega' \subset \Omega$ is a compact subset. It follows that \eqref{eq:localSolnStuff} can be rewritten over $\Omega'$ as
\begin{align*}
\tilde u \in \mathbb{K}(\varphi\Phi): \int_{\Omega'} a(\theta_1)\grad \tilde u \cdot \grad(\tilde u-v) &\leq \int_{\Omega'} \tilde f(\tilde u - v) \quad \forall v \in \mathbb{K}(\varphi\Phi).
\end{align*}
Note that from \eqref{eq:localSolnStuff}, the transformed function $w :=  \varphi\Phi-\tilde u$ satisfies the following variational inequality with zero lower obstacle:
\begin{equation}\label{eq:localSolnStuff2}
w \geq 0 : \langle A_\theta w -(A_\theta(\varphi\Phi)- \tilde f), w - v \rangle \leq 0 \quad\text{$\forall v \in H_0^1(\Omega) : v \geq 0$}.
\end{equation}
Now, let us for now assume that
\begin{equation}\label{eq:assumed1}
A_\theta(\varphi\Phi) \in L^2(\Omega) \quad \text{and} \quad \tilde f \in L^2(\Omega).
\end{equation}
Then the Lewy--Stampacchia inequality \cite[Theorem 3.3, \S 5:3]{Rodrigues} holds for this problem on $\Omega$, which reads $A_\theta (\varphi\Phi) -\tilde f \leq  A_\theta w \leq \max(A_\theta(\varphi\Phi) -\tilde f, 0)$. Transforming back, 
\begin{equation}\label{eq:eqForLS}
\min(\tilde f, A_\theta(\varphi\Phi)) \leq  A_\theta \tilde u  \leq \tilde f\qquad \text{a.e. in $\Omega$},
\end{equation}
showing that $A_\theta \tilde u \in L^2(\Omega)$. We need an elliptic regularity result to deduce from this information that $\tilde u$ belong to $H^2(\Omega)$. For this purpose, let us, again for now, assume that for all $i=1, \dotsc, n$, 
\begin{equation}\label{eq:assumed2}
a'(\theta_1)\partial_{x_i}\theta_1 \in L_\loc^q(\Omega) \text{ with } q> n.
\end{equation}
In this case, we may invoke the elliptic regularity of \cite[Lemma 7.1, \S 3]{LSUElliptic} for the operator $A_\theta$ acting on $\tilde u$:
\begin{align*}
\norm{\tilde u}{H^2(\Omega')} 
&\leq C\left(\norm{A_\theta \tilde  u}{L^2(\Omega')}^2 + \norm{\tilde f}{L^2(\Omega')}^2 + \norm{A_{\theta}(\varphi\Phi)}{L^2(\Omega')}^2\right),
\end{align*}
and we have just argued that the right-hand side is bounded. Finally, taking another subset
\[\Omega'' \subset\subset \Omega' \subset\subset \Omega,\]
supposing that $\varphi \equiv 1$ on $\Omega''$, the local $H^2$ regularity of the lemma as claimed follows from the trivial estimate
%
\begin{align*}
\norm{\tilde u}{H^2(\Omega')} \geq \norm{\tilde u}{H^2(\Omega'')} = \norm{u}{H^2(\Omega'')}.
\end{align*}
It remains for us to verify \eqref{eq:assumed1} and \eqref{eq:assumed2}. Let us consider the two cases of the lemma separately. 
 \medskip

\noindent (i) Suppose that $n \leq 3$.  Writing
\[A_\theta(\varphi\Phi) = \varphi A_\theta\Phi +\Phi A_\theta\varphi + 2a(\theta_1)\grad \theta_1 \grad \varphi,\] 
using the fact $\Phi, \theta_1 \in H^2_{\loc}(\Omega)$ and the embedding $H^2(\Omega') \cts W^{1,6}(\Omega') \cts C^{0,\alpha}(\Omega')$ (because of the low dimension), it follows that $A_\theta (\varphi\Phi) \in L^2(\Omega)$ (not just locally since $\varphi$ has compact support).  We also see that $\tilde f$ is bounded in $L^2(\Omega)$ due to the smoothness of $\varphi$. Therefore, we have \eqref{eq:assumed1}. Observe that $a'(\theta_1)\partial_{x_i}\theta_1 \in L_\loc^6(\Omega)$ for all $i=1, \dotsc, n$, and $6>n$ by assumption, so yielding \eqref{eq:assumed2}.

Now suppose instead that $a' \equiv 0$. Then \eqref{eq:assumed1} clearly holds after realising that $A_\theta(\varphi\Phi) = -a\Phi \Delta \varphi - a\varphi \Delta \Phi - 2a\grad\varphi\grad\Phi$ and \eqref{eq:assumed2} is redundant.


\medskip

\noindent (ii) If  $h_i \in L^p(\Omega)$ for $p > n$ and $i=1,2$, an application of \cite[Theorem 6.1, \S 7]{LSUElliptic} yields $\theta_i \in  C^{0,\gamma}(\bar\Omega) \cap C^{1,\alpha}(\Omega)$. Thanks to this and the interior $H^2$ regularity, $A_\theta(\varphi\Phi)$ and $\tilde f$ (and hence $A_\theta w$) remain bounded in $L^2(\Omega)$ (also due to the interior $H^2(\Omega)$ regularity). Furthermore, $a'(\theta_1)\partial_{x_i}\theta_1 \in L^q_{\loc}(\Omega)$ for $i=1, \dotsc, n$ and for any $q$ so both \eqref{eq:assumed1} and \eqref{eq:assumed2} hold.
\end{proof}
With this, we have proved all but item (iii) of Theorem \ref{thm:existence}, which follows immediately by the same reasoning and applications of elliptic regularity.

\begin{prop}\label{prop:ls}Let the assumptions of the previous lemma hold. The Lewy--Stampacchia inequality for $u$ is
\[\min(f, A_\theta\Phi) \leq  A_\theta u  \leq f\qquad \text{a.e. in $\Omega$}.\]
\end{prop}
\begin{proof}
Restricting \eqref{eq:eqForLS} to $\Omega''$ where $\varphi\equiv 1$,  we find the desired inequality but a.e.  in $\Omega''$. Since $\Omega''$ is arbitrary, it also holds almost everywhere in $\Omega$.
\end{proof}

\subsection{Identification of the characteristic function}\label{sec:ellipticIdChar}
We come now to the conclusion of the proof of Theorem \ref{thm:existenceRegularity}, which in fact holds for \emph{any} solution of \eqref{eq:equation} obtained as a result of our approximation process (we emphasise this because uniqueness of solutions is not known). We need some preliminary results in the context of the regularised problem \eqref{eq:regularised} first. 

Let us observe that since the interior regularity results of \S \ref{sec:localRegularity} also apply to solutions of the regularised problem \eqref{eq:regularised}, we obtain uniform boundedness in $H^2_{\loc}$ of the regularised solutions and therefore we can supplement the $H^1(\Omega)$ convergences of \eqref{eq:h1converences}. Indeed, 
\begin{equation*}
\begin{aligned}
\theta_i^\epsilon &\weaklyto \theta_i \qquad&&\text{in $H^2_{\loc}(\Omega)$,} \\
u^\epsilon &\weaklyto u &&\text{in $H^2_{\loc}(\Omega)$,}\\
\Phi^\epsilon &\weaklyto \Phi &&\text{in $H^2_{\loc}(\Omega)$.}
\end{aligned}
\end{equation*}
We enforce the assumptions of Theorem \ref{thm:existenceRegularity} and start by defining
\[A_\theta^\epsilon v := -\grad \cdot (a(\theta_1^\epsilon)\grad v).\]
\begin{lem}\label{lem:convergenceOfA}
If \eqref{ass:onCoefficientDeriv} holds, then $A_\theta^\epsilon u^\epsilon \weaklyto A_\theta u$ in $L^2_{\loc}(\Omega)$.
\end{lem}
\begin{proof}
We need to pass to the limit in $-A_\theta^\epsilon u^\epsilon = a(\theta_1^\epsilon)\Delta u^\epsilon + a'(\theta_1^\epsilon)\grad \theta_1^\epsilon\grad u^\epsilon$. The strong convergence results in \S \ref{sec:passageLimit} imply, via a Lebesgue's dominated convergence theorem argument, that $a(\theta_1^\epsilon) \to a(\theta_1)$ in $L^2(\Omega)$ and boundedness in $H^2_{\loc}(\Omega)$ of $u^\epsilon$ implies the weak convergence of $\Delta u^\epsilon$ in $L^2_{\loc}(\Omega)$. This handles the first term.

Since $\theta_1^\epsilon \to \theta_1$ strongly in $H^1(\Omega)$, it follows by the same logic as in the proof of Proposition \ref{prop:existenceForRegularisedProblem} that $a'(\theta_1^\epsilon)\grad \theta_1^\epsilon \to a'(\theta_1)\grad \theta_1$ in $L^2(\Omega)^n$. This, combined with the strong convergence of $u^\epsilon$ in $H^1_0(\Omega)$ is enough to prove the result. 
\end{proof}

We are now in position to establish the proof of Theorem \eqref{thm:existenceRegularity} and hence the existence of regular solutions to the elliptic problem.

\begin{proof}[Proof of Theorem \ref{thm:existenceRegularity}]
Define $\hat \chi_\epsilon := \chi_{\{\Phi^\epsilon=u^\epsilon\}}$ and set $\chi^\epsilon = \chi_\epsilon(\Phi^\epsilon-u^\epsilon)$. Observe that we always have
\begin{equation}\label{eq:inequalityToPTL}
\chi^\epsilon \geq \hat \chi_\epsilon
\end{equation}
because both functions agree on the coincidence set. Let $\hat \chi$ stand for the weak-* limit of $\{\hat \chi_\epsilon\}$ (after relabelling the subsequence):
\[\hat \chi_\epsilon \weaklystar \hat \chi \quad \text{in $L^\infty(\Omega)$.}\]
Passing to the limit in \eqref{eq:inequalityToPTL}, we obtain
\begin{equation}\label{eq:inequalityToCompare}
\chi \geq \hat \chi.
\end{equation}
Taking into account Lemma \ref{lem:XiHeaviside}, it remains to show that $\hat \chi \geq \chi_{\{\Phi=u\}}$.
Recall \eqref{eq:localSolnStuff2}. We have seen that the source term in it belongs to $L^2(\Omega)$ and therefore, by \cite[Theorem 2.7, \S 5:2]{Rodrigues}, on the non-coincidence set, $w=\varphi\Phi - \varphi u$ solves the associated equation
\[A_\theta w = A_\theta(\varphi\Phi)-F \text{ on $\{w > 0\}$}.\]
This then implies $A_\theta (\varphi u) = F \text{ on $\{\varphi\Phi > \varphi u \}$}$ and hence, recalling that $\varphi \equiv 1$ on $\Omega''$, via Stampacchia's lemma (since $u \in H^2_\loc(\Omega)$)
\[A_\theta u = f \text{ on $\{\Phi > u \} \cap \Omega''$}.\]

Stampacchia's lemma can again be invoked to yield $A_\theta u = A_\theta\Phi$ in $\{u=\Phi\} \cap \Omega''$. These two facts, along with analogous arguments for $u^\epsilon$,  are enough to obtain the equations
\begin{equation*}
\begin{aligned}
A_\theta u + (f-A_\theta\Phi)\chi_{\{\Phi=u\}} &= f &&\text{on $\Omega''$},\\
\nonumber A^\epsilon_\theta u^\epsilon + (f-A^\epsilon_\theta\Phi^\epsilon)\hat \chi_\epsilon &= f&&\text{on $\Omega''$}.
\end{aligned}
\end{equation*}
%
%
%
Let us pass to the limit in the latter equation and compare the end result to the former. We see that, since
\begin{align*}
-A^\epsilon_\theta\Phi^\epsilon &= a(\theta_1^\epsilon)\Delta\Phi^\epsilon + a'(\theta_1^\epsilon)\grad \theta_1^\epsilon\grad\Phi^\epsilon\\
&= -a(\theta_1^\epsilon)(g+\alpha(\theta_1^\epsilon-\theta_2^\epsilon)\chi^\epsilon) + a'(\theta_1^\epsilon)\grad \theta_1^\epsilon\grad\Phi^\epsilon,
\end{align*}
we get on the one hand, by using $\chi^\epsilon\hat \chi_\epsilon = \hat \chi_\epsilon$,
\begin{align*}
(f-A^\epsilon_\theta\Phi^\epsilon)\hat \chi_\epsilon 
&= (f-a(\theta_1^\epsilon)g -\alpha a(\theta_1^\epsilon)(\theta_1^\epsilon-\theta_2^\epsilon)+ a'(\theta_1^\epsilon)\grad \theta_1^\epsilon\grad\Phi^\epsilon)\hat \chi_\epsilon \\
&\weaklyto (f-a(\theta_1)g -\alpha a(\theta_1)(\theta_1-\theta_2) + a'(\theta_1)\grad \theta_1\grad\Phi)\hat \chi\\
&= (f+ a(\theta_1)\Delta \Phi + \alpha a(\theta_1)(\theta_1-\theta_2)\chi  -\alpha a(\theta_1)(\theta_1-\theta_2) + a'(\theta_1)\grad \theta_1\grad\Phi)\hat \chi\\
&= (f - A_\theta\Phi + \alpha a(\theta_1)(\theta_1-\theta_2)\chi  -\alpha a(\theta_1)(\theta_1-\theta_2))\hat \chi.
\end{align*}
On the other hand, by using the equation itself,
\begin{align*}
(f-A^\epsilon_\theta\Phi^\epsilon)\hat \chi_\epsilon = f-A^\epsilon_\theta u^\epsilon  \weaklyto f-A_\theta u = (f-A_\theta \Phi)\chi_{\{\Phi=u\}}
\end{align*}
by Lemma \ref{lem:convergenceOfA}. Therefore,
\[(f - A_\theta\Phi + \alpha a(\theta_1)(\theta_1-\theta_2)\chi  -\alpha a(\theta_1)(\theta_1-\theta_2))\hat \chi = (f-A_\theta \Phi)\chi_{\{\Phi=u\}}\]
whence
\begin{align*}
(f - A_\theta\Phi)(\hat \chi-\chi_{\{\Phi=u\}})  = \alpha a(\theta_1)(\theta_1-\theta_2)(1-\chi)\hat \chi.
\end{align*}
It follows that if $\theta_1 \geq \theta_2$, the right-hand side is non-negative, which then along with the non-degeneracy condition \eqref{ass:forRegularity} implies that $\hat \chi \geq \chi_{\{\Phi=u\}}$. This combined with \eqref{eq:inequalityToCompare} gives $\chi \geq \chi_{\{\Phi=u\}}.$ Thanks to Lemma \ref{lem:XiHeaviside} (which gave us the reverse inequality), we can conclude that $\chi = \chi_{\{\Phi=u\}}$ a.e. in $\Omega''$, and once again the arbitrariness of $\Omega''$ yields the a.e. equality in the whole of $\Omega$.
\end{proof}

\subsection{Uniqueness}\label{sec:uniquenessElliptic}
We finalise this section by proving uniqueness of solutions by using the continuous dependence result of Proposition \ref{prop:L1ctsDepElliptic}.

\begin{proof}[Proof of Theorem \ref{thm:uniquenessElliptic}]
Let $(\theta_1, \theta_2, \Phi, u, \chi)$ and $(\hat\theta_1, \hat\theta_2, \hat\Phi, \hat u, \hat\chi)$ denote two regular solutions corresponding to the same data.  Applying the `strong coercivity' condition \eqref{ass:forUniqueness0} to the $L^1$-continuous dependence estimate \eqref{eq:L1CtsDependence} in our setting, we see that
\begin{align}
\norm{\theta_1-\hat \theta_1}{L^1(\Omega)} +
\norm{\theta_2-\hat \theta_2}{L^1(\Omega)} &\leq \frac{(M-m)(b_1+b_2)}{\gamma_0}\norm{\chi-\hat\chi}{L^1(\Omega)}.\label{eq:eqToCompare}
\end{align}
We use this to estimate the difference of the two obstacles:
\begin{align}
\nonumber \norm{\Delta \hat \Phi - \Delta \Phi}{L^1(\Omega)} &\leq \alpha\norm{(\hat\theta_1-\hat\theta_2)\hat\chi - (\theta_1-\theta_2)\chi}{L^1(\Omega)}\\
\nonumber &= \alpha\norm{(\hat\theta_1-\hat\theta_2- (\theta_1-\theta_2))\hat\chi+ (\theta_1 -\theta_2 )(\hat\chi - \chi)}{L^1(\Omega)}\\
\nonumber &\leq \alpha\norm{\theta_1 - \hat \theta_1}{L^1(\Omega)} + \alpha\norm{\theta_2 - \hat \theta_2}{L^1(\Omega)} +  \alpha (M-m)\norm{\chi-\hat\chi}{L^1(\Omega)}\\
\nonumber &\leq  \frac{\alpha(M-m)(b_1+b_2)}{\gamma_0}\norm{\chi-\hat\chi}{L^1(\Omega)} +  \alpha (M-m)\norm{\chi-\hat\chi}{L^1(\Omega)}\\
&=\alpha(M-m)\left(\frac{b_1+b_2}{\gamma_0}+ 1\right)\norm{\chi-\hat\chi}{L^1(\Omega)}.\label{eq:toComp10}
\end{align}
We aim to estimate the term involving the difference of the characteristic function by the continuous dependence estimate for characteristic functions in \cite[Theorem 4.7, \S 5:4]{Rodrigues} for the two obstacle problems satisfied by $u$ and $\hat u$. Indeed, the transformed functions $w := \Phi-u$ and $\hat w:= \hat \Phi - \hat u$ satisfy variational inequalities with zero lower obstacles and source terms $-a\Delta\Phi -f$ and $-a\Delta \hat\Phi - f$ respectively, and observing that
\[f+a\Delta\Phi = f-ag-a\alpha(\theta_1-\theta_2)\chi \geq f-ag-a\alpha(M-m),\]
the non-degeneracy condition \eqref{ass:strongNonDegForUniqueness0} implies that there exists a constant $\lambda$ with 
\[f+a\Delta\Phi \geq \lambda > a\alpha(M-m)\left(1 + \frac{b_1+b_2}{\gamma_0}\right).\]
This implies that the non-degeneracy condition of the cited theorem is valid and it can be applied to yield
\begin{align*}
\lambda\norm{\chi-\hat\chi}{L^1(\Omega)} &\leq a\norm{\Delta \hat \Phi - \Delta \Phi}{L^1(\Omega)}\\
 &\leq a\alpha(M-m)\left(1 + \frac{b_1+b_2}{\gamma_0}\right)\norm{\chi-\hat\chi}{L^1(\Omega)}
\end{align*}
where we used \eqref{eq:toComp10}. This shows that $\chi = \hat \chi$, in turn giving $\Phi = \hat \Phi$ and from \eqref{eq:eqToCompare}, $\theta_i = \hat \theta_i$. From this, uniqueness of $u$ follows easily. 

\end{proof}

\section{The quasistatic (evolutionary) problem}\label{sec:parabolic}
We come now to the study of the evolutionary problem \eqref{eq:parabolicEquation}. The major results are stated in the next two sections. It is important to note that our results in \S \ref{sec:parabolicMain} for  the continuous problem \eqref{eq:parabolicEquation} rely on the assumption that $a' \equiv 0$, which is not necessary for the results in \S \ref{sec:parabolicSemi} for the time-discretised version.

\subsection{Main results on the evolutionary problem}\label{sec:parabolicMain}
Let us define the usual spacetime cylinder $Q:=[0,T]\times\Omega$ and its lateral boundary $\Sigma := [0,T]\times \partial\Omega$. The main result we are able to prove is the following. In it,  regarding the assumption \eqref{ass:strongNonDegeneracy}, see Remark \ref{rem:onStrongNonDegAss}. As in the elliptic case, we shall need the technical assumption 
\begin{equation}
\kappa:=\kappa_1 = \kappa_2, \quad c_2\geq c_1, \quad h_1 \geq h_2 \geq 0, \quad \theta_{10} \geq \theta_{20} \geq 0,\label{ass:parabolicAssConsts}
\end{equation}
in order to enforce a favourable sign condition on the difference of the temperatures.
\begin{theorem}[\textsc{Existence}]\label{thm:existenceParabolic}
Let 
\begin{equation}\label{ass:aPrimeZero}
a' \equiv 0,
\end{equation}
and take $f, g, h_1, h_2 \in L^\infty(0,T;L^2(\Omega))$ with $h_1 \in L^1(0,T;L^\infty(\Omega))$, $\theta_{i0} \in H^1(\Omega)$ with $\theta_{10} \in L^\infty(\Omega)$,  \eqref{ass:onConstants},  \eqref{ass:parabolicAssConsts}, and
\begin{align}
&f-ag > 2a\alpha \norm{h_1}{L^1(0,T;L^\infty(\Omega))} + 2a\alpha\norm{\theta_{10}}{L^\infty(\Omega)}\quad\text{a.e. in $Q$}.\label{ass:strongNonDegeneracy}
\end{align}
Then there exists a solution
\begin{align*}
&\theta_i \in L^\infty(0,T;H^1(\Omega)) \cap L^2(0,T;H^2_{\loc}(\Omega)) \text{ with } \partial_t \theta_i \in L^\infty(0,T;H^1(\Omega)^*) \cap L^2(0,T;L^2(\Omega)),\\
&\Phi \in L^\infty(0,T;H^1_0(\Omega)) \cap  L^2(0,T;H^2_{\loc}(\Omega)),\\
&u \in L^\infty(0,T;H^1_0(\Omega))\cap L^2(0,T;H^2_{\loc}(\Omega)),
\end{align*}
to the system
\begin{align*}
&\text{for $i=1, 2$:} &\partial_t \theta_i - \kappa_i \Delta \theta_i + c_i\theta_i &= h_i + (-1)^ib_i(\theta_1-\theta_2)\chi_{\{u=\Phi\}} &&\text{in $Q$},\\
&&\partial_n \theta_i &= 0 &&\text{on $\Sigma$},\\
&& \theta_i(0) &= \theta_{i0} &&\text{in $\Omega$},\\
\nonumber &\text{for a.e. $t \in (0,T)$:}\\
&&-\Delta \Phi(t) = \alpha(&\theta_1(t) - \theta_2(t))\chi_{\{u(t)=\Phi(t)\}} +g(t) &&\text{in $\Omega$},\\
&&\Phi(t) &= 0 &&\text{on $\partial\Omega$},\\
&&u(t) \in \mathbb{K}(\Phi(t)),\quad -a\Delta u(t)&\leq  f(t), \quad (-a\Delta  u(t)-f(t))(u(t)-\Phi(t)) = 0 &&\text{in $\Omega$},\\
&&u(t) &=0 &&\text{on $\partial\Omega$}.
\end{align*}
(This is \eqref{eq:parabolicEquation} with $A_\theta \equiv -a\Delta$).
\end{theorem}
We think of solutions given by this theorem as \textbf{regular solutions} in analogy with the elliptic case because of the appearance of the characteristic function in the solution concept. Now, let us show that the temperatures can be bounded from above and below in terms of the data in the following analogue of Proposition \ref{prop:LInftyEstimates}.
\begin{prop}\label{prop:LInftyEstimatesPar}Let \eqref{ass:onConstants} hold and let $\theta_i$ be a solution of \eqref{eq:parabolicEquation1}--\eqref{eq:parabolicEquation3}. We have
\[\theta_1, \theta_2 \geq \min\left(\frac{1}{c_1}\essinf_Q h_1,  \frac{1}{c_2}\essinf_Q h_2, \essinf_\Omega \theta_{10}, \essinf_\Omega \theta_{20} \right) =: l\]
and
\[\theta_1, \theta_2 \leq \max\left(\frac{1}{c_1}\esssup_Q h_1, \frac{1}{c_2}\esssup_Q h_2, \esssup_\Omega \theta_{10}, \esssup_\Omega \theta_{20}  \right) =: L.\]
\end{prop}
\begin{proof}
Testing the $\theta_i$ equation with $(\theta_i-L)^+$ for a constant $L>0$ to be fixed, we find after writing $c_i\theta_i = c_i(\theta_i -L) + c_iL$,
\begin{align*}
\frac{1}{2}\frac{d}{dt}\int_\Omega |(\theta_i-L)^+|^2 + \int_\Omega \kappa_i |\grad (\theta_i-L)^+|^2 + c_i|(\theta_i-L)^+|^2 &= \int_\Omega (h_i-c_iL)(\theta_i-L)^+ + (-1)^ib_i(\theta_1-\theta_2)\chi(\theta_i-L)^+.
\end{align*}
Assume that $b_1 > b_2$. Manipulating in the same as in the proof of Proposition \ref{prop:LInftyEstimates},  
 we see that
\begin{align*}
&\int_0^T \int_\Omega \kappa_1 |\grad (\theta_1-L)^+|^2 + \kappa_2 |\grad (\theta_2-L)^+|^2 + c_1|(\theta_1-L)^+|^2  + \left(c_2-\frac{b_1-b_2}{4}\right)|(\theta_2-L)^+|^2\\
 &\leq \int_0^T\int_\Omega (h_1-c_1L)(\theta_1-L)^+ +  (h_2-c_2L)(\theta_2-L)^+ 
+ \frac{1}{2}\int_\Omega |(\theta_{10}-L)^+|^2 + |(\theta_{20}-L)^+|^2.
\end{align*}
This yields the result as before.
\end{proof}
\begin{remark}\label{rem:linftyBoundOnDifferenceThetas}
Suppose that \eqref{ass:onConstants}, \eqref{ass:onConstantsForRegularity} and
\begin{equation*}
\text{$c:= c_1 =c_2$, $h_1-h_2 \in L^\infty(Q)$, and $\theta_{10}-\theta_{20} \in L^\infty(\Omega)$.}
\end{equation*}
Then the difference $\gamma = \theta_1-\theta_2$ satisfies
\[\partial_t \gamma - \kappa\Delta \gamma + (c+(b_1+b_2)\tilde \chi)\gamma = h_1-h_2.\]
Define $K_1 := \norm{h_1-h_2}{L^\infty(Q)}$ and $K_2:=\norm{\theta_{10}-\theta_{20}}{L^\infty(\Omega)}$ and setting $v(t):= K_1t + K_2$, we see that
\[\partial_t (\gamma-v) - \kappa\Delta (\gamma-v) + (c+(b_1+b_2)\tilde \chi)\gamma = h_1-h_2-K_1.\]
Testing now with $(\gamma-v)^+$ and noting that the last term on the left-hand side of the above is non-negative, we eventually obtain 
\[\norm{\theta_1-\theta_2}{L^\infty(Q)} \leq T\norm{h_1-h_2}{L^\infty(Q)} + \norm{\theta_{10}-\theta_{20}}{L^\infty(\Omega)}.\]
This means that if the initial data are sufficiently close to each other, then no matter how dissimilar $h_1$ and $h_2$, for small time, $\theta_1$ and $\theta_2$ are close. 
\end{remark}
Theorem \ref{thm:existenceParabolic} guarantees that $\theta_i \in C^0([0,T];H^{1\slash 2}(\Omega))$ by the standard Sobolev--Bochner embeddings. Under additional regularity on the data, we can obtain continuity of the membrane, mould as well as the characteristic function in the system as the next theorem shows.
\begin{theorem}[\textsc{Continuity in time}]\label{thm:continuityParabolic}
Let the assumptions of Theorem \ref{thm:existenceParabolic} hold  and let $f, g \in C^{0,\gamma}((0,T);L^1(\Omega))$ for some $\gamma \in (0,1]$. Then
\begin{align*}
&\chi_{\{u=\Phi\}} \in C^0((0,T);L^p(\Omega)) \text{ for all $p<\infty$}
\end{align*}
and for $s,t \in (0,T)$, the following continuous dependence estimate holds:
\begin{align*}
\norm{\chi_{\{u(t)=\Phi(t)\}}-\chi_{\{u(s)=\Phi(s)\}}}{L^1(\Omega)}
&\leq C\big(\norm{f(t)-f(s)}{L^1(\Omega)} + a\norm{g(t)-g(s)}{L^1(\Omega)}\\
&\quad\quad\quad + a\alpha\norm{(\theta_1(t)-\theta_2(t))-(\theta_1(s)-\theta_2(s))}{L^1(\Omega)}\big).
\end{align*}
If also
\begin{equation*}
f, g \in C^0((0,T);L^r(\Omega)) \text{ for some $r > 1$,}
\end{equation*}
then for all $q < 2n\slash(n-2)$ and $\epsilon >0$,
\begin{align*}
\Phi \in C^0((0,T);W^{2,\min(r,q)}_{\loc}(\Omega)),\quad \text{ and} \quad 
u \in C^0((0,T);W^{2,\min(r,q)-\epsilon}_{\loc}(\Omega)).
\end{align*}
\end{theorem}

Finally, the analogous uniqueness result to Theorem \ref{thm:uniquenessElliptic}, is given by the following.

\begin{theorem}[\textsc{Uniqueness}]\label{thm:uniquenessPar}
Let $a'\equiv 0$, $h_1, h_2 \in L^\infty(Q)$, $\theta_{10}, \theta_{20} \in L^\infty(\Omega)$, \eqref{ass:onConstants} and suppose that
\begin{align}
&f>ag+ a\alpha (L-l)\left(2 + \gamma_0^{-1}(b_1+b_2)\right)\quad \text{a.e. in $Q$,}\label{ass:strongNonDegForUniqueness0Par}
\end{align}
where $\gamma_0$ is as in \eqref{ass:forUniqueness0} from Theorem \ref{thm:uniquenessElliptic}. Then the solution of \eqref{eq:parabolicEquation} is unique.\end{theorem}
\subsection{The semi-discretised problem}\label{sec:parabolicSemi}
Our method of proof for the solvability of the evolutionary problem is through a semi-discretisation in time of the problem \eqref{eq:parabolicEquation}. Observe that the results in this section do \textit{not} require $a' \equiv 0$ in contrast to the previous section.

For this purpose, let $N \in \mathbb{N}$, $\tau:=T\slash N$ and for $j=0, 1, ..., N$, define $t_j:=jh$. Setting $I_{k} = [t_{k-1}, t_k)$ for $k=1, \dotsc, N$ enables us to divide $[0,T]$ into $N$ subintervals of length $\tau$. One should bear in mind that all of these objects depend on $N$ but we shall omit this dependence for clarity. Regarding the approximation of the source terms, we use the zero-order Cl\'ement quasi-interpolants, i.e.
\[f^N(t) := \sum_{k=1}^N f^k \chi_{I_k}(t)\quad\text{where}\quad f^k := \frac 1\tau\int_{I_k}f(s)\;\mathrm{d}s,\]
and we define $h_i^N$ and $g^N$ similarly.
We consider for $k = 1, \dotsc, N$ the following semi-discretised problem  associated to \eqref{eq:parabolicEquation}:
\begin{equation}\label{eq:parabolicEquationDiscWeaker}
\begin{aligned}
\text{for $i=1, 2$:} \qquad \int_\Omega \left(\frac{\theta_i^k-\theta_i^{k-1}}{\tau}\right)\eta + \kappa_i \grad \theta_i^k\cdot \grad \eta + c_i\theta_i^k\eta &= \int_\Omega (h_i^k + (-1)^ib_i(\theta_1^k-\theta_2^k)\chi_k)\eta &&\forall \eta \in H^1(\Omega),\\
\int_\Omega \grad \Phi^k\cdot \grad \xi &= \int_\Omega (\alpha(\theta_1^k - \theta_2^k)\chi_k +g^k)\xi &&\forall \xi \in H^1_0(\Omega),\\
u^k \in \mathbb{K}(\Phi^k) \::\: \int_\Omega a(\theta_1^k)\grad u^k \cdot \grad(u^k -v) &\leq \int_\Omega f^k(u^k -v) &&\forall v \in \mathbb{K}(\Phi^k),\\
\end{aligned}
\end{equation}
where $\chi_k \in 1-H(\Phi^k-u^k)$ and we set  $\theta_i^0 = \theta_{i0}$. Thanks to \S \ref{sec:elliptic}, we have at our disposal the well posedness of this system as the following propositions show.
%
\begin{prop}[\textsc{Existence of weak solutions for the semi-discretised problem}]\label{prop:existenceSemiDisc}Let $f, g, h_1, h_2 \in L^2(0,T;L^2(\Omega))$, $\theta_{i0} \in L^2(\Omega)$ and let \eqref{ass:onConstants} hold. Then for each $k \geq 1$,  \eqref{eq:parabolicEquationDiscWeaker} has a solution
\[(\theta_1^k, \theta_2^k, \Phi^k, u^k) \in (H^1(\Omega)\cap H^2_{\loc}(\Omega))^2 \times H^1_0(\Omega) \times (H^1_0(\Omega) \cap H^2_{\loc}(\Omega))\]
with $\chi_k \in 1-H(\Phi^k-u^k)$.

\medskip

\noindent Furthermore, under \eqref{ass:onCoefficientDeriv}, 
\begin{enumerate}[(i)]
\item if $n \leq 3$ or $a' \equiv 0$, then $u^k \in H^2_{\loc}(\Omega)$,
\item if  $h_1, h_2 \in L^2(0,T;L^p(\Omega))$ for $p > n$, then $u^k \in H^2_{\loc}(\Omega)$ and $\theta_1^k, \theta_2^k \in C^{1,\alpha}(\Omega)$.
\end{enumerate}

\medskip

\noindent In addition,  if $\kappa:=\kappa_1 = \kappa_2$, $h_1, h_2 \geq 0$ and $\theta_{10}, \theta_{20} \geq 0$, then $\theta_1^k, \theta_2^k \geq 0$, and if  \eqref{ass:parabolicAssConsts}	 holds, then $\theta_1^k \geq \theta_2^k \geq 0$.
\end{prop}
\begin{proof}
The system \eqref{eq:parabolicEquationDiscWeaker} can be seen to be of the form considered in \S \ref{sec:elliptic} since the equation for $\theta_i^k$ can be rewritten as
\begin{equation*}
\begin{aligned}
- \kappa \Delta \theta_i^k + (c_i+\tau^{-1})\theta_i^k &= h_i^k+\tau^{-1}\theta_i^{k-1} + (-1)^ib_i(\theta_1^k-\theta_2^k)\chi_k.
\end{aligned}
\end{equation*}
Observe that coercivity for the elliptic operator clearly follows by \eqref{ass:onConstants} (the extra $\tau^{-1}$ term does not hinder). Then we can simply apply Theorem \ref{thm:existence} for the existence and the $H^2_{\loc}$ regularity. The assumption on the $h_i$ implies that $h_1^k \geq h_2^k$ for each $k$ and this, along with the assumption on the initial  data, gives us $\theta_1^1 \geq \theta_2^1$ by Proposition \ref{prop:comparisonPrinciple}. Iterating this argument gives the result for all $k$. The non-negativity follows by similar reasoning and repeated applications of \eqref{eq:nonnegativityOfTemps}.
\end{proof}
A direct application of Theorem \ref{thm:existenceRegularity} gives the following regularity result.
\begin{prop}[\textsc{Regular solutions for the semi-discretised problem}]\label{prop:regularitySemiDisc}
Let the assumptions of the previous proposition hold as well as \eqref{ass:onCoefficientDeriv}, \eqref{ass:parabolicAssConsts} and
\begin{equation}
\text{if  $n > 3$ and $a' \not\equiv 0$, $h_1, h_2 \in L^2(0,T;L^p(\Omega))$ for $p>n$}.
\end{equation}
If furthermore 
\begin{align}
&f^k+\grad \cdot (a(\theta_1^k)\grad \Phi^k) > 0 \quad \text{a.e. in $\Omega$},\label{ass:forRegularitySemiDisc}
\end{align}
then 
\[\chi_k = \chi_{\{\Phi^k = u^k\}}.\]
\end{prop}
Let us now give a condition on the data under the setting  $a' \equiv 0$ that implies the non-degeneracy assumption of the previous proposition.
\begin{lem}\label{lem:conditionForNonDegSemiDisc}Let $a' \equiv 0$, $h_1 \in L^1(0,T;L^\infty(\Omega))$, $\theta_{10} \in L^\infty(\Omega)$ and let \eqref{ass:onConstants} and \eqref{ass:parabolicAssConsts} hold. If
\begin{equation}\label{ass:conditionToGetCharPar}
f-ag > \alpha a\norm{h_1}{L^1(0,T;L^\infty(\Omega))} + \alpha a\norm{\theta_{10}}{L^\infty(\Omega)}\quad \text{a.e. in $Q$,}
\end{equation}
then assumption \eqref{ass:forRegularitySemiDisc} is met. 
\end{lem}
\begin{proof}
Let us first prove that for each $k$, 
\begin{align}
\norm{\theta_1^k}{L^\infty(\Omega)} &\leq \norm{h_1}{L^1(0,T;L^\infty(\Omega))} + \norm{\theta_{10}}{L^\infty(\Omega)}\label{eq:boundOnTheta1K}.
\end{align}
Using the $L^\infty$ estimate on the temperature in Lemma \ref{lem:LinftyBoundTheta1}, we obtain for $k \geq 1$,
\begin{align*}
\norm{\theta_1^k}{L^\infty(\Omega)} &\leq \frac{\norm{h_1^k+\tau^{-1}\theta_1^{k-1}}{L^\infty(\Omega)}}{c_1 + \tau^{-1}}\\
&\leq \frac{\tau}{\tau c_1 + 1}\norm{h_1^k}{L^\infty(\Omega)} + \frac{1}{\tau c_1 + 1}\norm{\theta_1^{k-1}}{L^\infty(\Omega)},
\end{align*}
whence, solving the recurrence inequality,
\begin{align}
\nonumber \norm{\theta_1^k}{L^\infty(\Omega)} &\leq \sum_{j=1}^k \frac{\tau}{(\tau c_1+1)^j}\norm{h_1^{k+1-j}}{L^\infty(\Omega)} + \frac{1}{(\tau c+1)^k}\norm{\theta_{10}}{L^\infty(\Omega)}\\
\nonumber &\leq \sum_{j=1}^k \tau\norm{h_1^{k+1-j}}{L^\infty(\Omega)} + \norm{\theta_{10}}{L^\infty(\Omega)}\\
\nonumber &= \sum_{j=0}^{k-1} \tau\norm{h_1^{k-j}}{L^\infty(\Omega)} + \norm{\theta_{10}}{L^\infty(\Omega)}\\
\nonumber &\leq \sum_{j=0}^{N-1} \tau\norm{h_1^{k-j}}{L^\infty(\Omega)} + \norm{\theta_{10}}{L^\infty(\Omega)}\\
\nonumber &= \norm{h_1^N}{L^1(0,T;L^\infty(\Omega))} + \norm{\theta_{10}}{L^\infty(\Omega)}.
\end{align}
Then we simply use 
\begin{align*}
\norm{h_1^N}{L^1(0,T;L^\infty(\Omega))} 
&= \frac 1\tau\int_0^T\sum_{k=1}^N \norm{\int_{I_k} h_1(s)\;\mathrm{d}s}{L^\infty(\Omega)}\chi_{I_k}(t)\\
&\leq \frac 1\tau\sum_{k=1}^N \int_{I_k} \norm{h_1}{L^1(I_k; L^\infty(\Omega))}\\
&= \norm{h_1}{L^1(0,T;L^\infty(\Omega))}.
\end{align*}
Now,  note that the assumption \eqref{ass:conditionToGetCharPar} implies the existence of a constant $\mu$ such that
\[f-ag - \alpha a\norm{h_1}{L^1(0,T;L^\infty(\Omega))} - \alpha a\norm{\theta_{10}}{L^\infty(\Omega)} \geq \mu > 0.\]
The left-hand side of the desired inequality \eqref{ass:forRegularitySemiDisc} can be manipulated by plugging in the equation for $\Phi^k$, using the fact that the temperatures are non-negative and the definition of the Cl\'ement interpolants:
\begin{align}
\nonumber f^k + a\Delta  \Phi^k  &= f^k -ag^k -a\alpha(\theta_1^k-\theta_2^k)\chi^k\\
\nonumber &\geq \frac 1\tau \int_{I_k}f(s)-ag(s)\;\mathrm{d}s - a\alpha\norm{\theta_1^{k}}{L^\infty(\Omega)}\\
&\geq \mu\label{eq:estimateForDegeneracyInK}
\end{align}
where the final inequality follows by making use of \eqref{eq:boundOnTheta1K}.
\end{proof}
Theorem \ref{thm:existenceParabolic} relies on the identification of $\chi^k$ from the semi-discretised problem as the characteristic function (provided by Proposition \ref{prop:regularitySemiDisc}) in addition to the strong non-degeneracy assumption \eqref{ass:strongNonDegeneracy}.  The passage to the limit in the discretisation parameter and the obtainment of existence for the evolutionary model where only the weak solution of Proposition \ref{prop:existenceSemiDisc} is available (i.e. when non-degeneracy is not assured)  is a significant challenge and is still an open problem. We discuss some of the issues that arise in that case in \S \ref{sec:remarksOnNonRegularCase}. Now, let us proceed with proving the results stated in \S \ref{sec:parabolicMain}.

\subsection{Interpolants}\label{sec:interpolants}
For the time being, the assumptions of Proposition \ref{prop:existenceSemiDisc} are enforced so that from  \eqref{eq:parabolicEquationDiscWeaker}, we construct the piecewise constant interpolants
\[\theta_i^N(t,x):=\sum_{k=1}^N \theta_i^k(x) \chi_{I_k}(t)\qquad\text{and}\qquad \chi^N(t,x) := \sum_{k=1}^N \chi_k(x)\chi_{I_k}(t)\]
as well as the piecewise affine interpolant
\begin{align*}
\hat \theta_i^N(t) &:= \theta_{i0} + \int_0^t \sum_{k=1}^N \frac{\theta_i^k - \theta_i^{k-1}}{\tau}\chi_{I_k}(s)\;\mathrm{d}s \equiv \sum_{k=1}^N\left(\theta_{i}^{k-1}  + \frac{\theta_i^{k}-\theta_i^{k-1}}{\tau}(t-t_{k-1})\right)\chi_{I_k}(t).
\end{align*}
For the relation between $\hat \theta_i^N$ and $\theta_i^N$, see Lemma \ref{lem:differentInterpolants}. Defining $\Phi^N$ and $u^N$ in the same way as $\theta_i^N$, we find after multiplying each line of \eqref{eq:parabolicEquationDiscWeaker} by $\chi_{I_k}$ and summing that these quantities satisfy the following system: 
\begin{equation}\label{eq:parabolicNEquations}
\begin{aligned}
&\text{for $i=1, 2$:} &\partial_t \hat \theta_i^N - \kappa_i\Delta \theta_i^N + c_i\theta_i^N &= h_i^N + (-1)^ib_i(\theta_1^N-\theta_2^N)\chi^N&&\text{in $Q$},\\
&&\partial_n \theta_i^N &= 0&&\text{on $\Sigma$},\\
&&\hat\theta_i^N(0) &= \theta_{i0}&&\text{in $\Omega$},\\
&&-\Delta \Phi^N &= \alpha(\theta_1^N-\theta_2^N)\chi^N + g^N&&\text{in $Q$},\\
&&\Phi^N &= 0&&\text{on $\Sigma$},\\
&&\int_\Omega a(\theta_1^N)\grad u^N \grad(u^N-v) &\leq \int_\Omega f^N(u^N-v) \quad \forall v :=\textstyle \sum_{k=1}^Nv^k\chi_{I_k}, \quad v^k \in \mathbb{K}(\Phi^k),\\
&&u^N &= 0&&\text{on $\Sigma.$}
\end{aligned}
\end{equation}
We now look for uniform estimates on the solution in the semi-discretised problem \eqref{eq:parabolicEquationDiscWeaker} that will translate into estimates on the associated interpolants and then we will pass to the limit in \eqref{eq:parabolicNEquations}.

The following bound is trivial:
\[\norm{\chi^N}{L^\infty(Q)} = \esssup_{(t,x) \in Q}\left|\sum_{k=1}^N \chi_k(x)\chi_{I_k}(t)\right| \leq 1.\]
Estimates independent of $N$ on the other quantities require more work as one can see in the next three lemmas.
\begin{lem}\label{lem:discBounds}For $i=1,2$ and all $k$, the following bound holds uniformly in $N$:
\begin{align*}
\norm{\theta_i^k}{H^1(\Omega)} + \tau\sum_{k=1}^N\norm{\Delta \theta_i^k}{L^2(\Omega)}^2 + \frac{1}{\tau}\sum_{k=1}^N\norm{\theta_i^k-\theta_i^{k-1}}{L^2(\Omega)}^2 &\leq C.
\end{align*}
\end{lem}
\begin{proof}
Below, we use the notation $\lVert \cdot \rVert$ in place of $\lVert \cdot \rVert_{L^2(\Omega)}$ for ease of reading. Test the $\theta_i^k$ equation with $\theta_i^k$ and use the inner product identity $2(a-b,a) = \norm{a}{}^2 - \norm{b}{}^2 + \norm{a-b}{}^2$ to get
\begin{align*}
\frac{1}{2\tau}\left(\norm{\theta_i^k}{}^2 - \norm{\theta_i^{k-1}}{}^2 + \norm{\theta_i^k-\theta_i^{k-1}}{}^2\right) + \kappa_i\norm{\grad \theta_i^k}{}^2 + c_i\norm{\theta_i^k}{}^2  &\leq \norm{h_i^k}{}\norm{\theta_i^k}{} + \int_\Omega (-1)^ib_i(\theta_1^k-\theta_2^k)\theta_i^k\chi_k.
\end{align*}
Assuming first of all that $b_1 > b_2$, adding the inequalities for $i=1$ and $i=2$ and proceeding as in the proof of Proposition \ref{prop:auxiliaryExistence} to deal with the two integral terms, we obtain
\begin{align*}
&\frac{1}{2\tau} \left(\norm{\theta_1^k}{}^2 - \norm{\theta_1^{k-1}}{}^2 + \norm{\theta_2^k}{}^2 - \norm{\theta_2^{k-1}}{}^2  + \norm{\theta_1^k-\theta_1^{k-1}}{}^2  + \norm{\theta_2^k-\theta_2^{k-1}}{}^2\right)+ \kappa_1\norm{\grad \theta_1^k}{}^2 + c_1\norm{\theta_1^k}{}^2   \\
&\quad + \kappa_2\norm{\grad \theta_2^k}{}^2 + c_2\norm{\theta_2^k}{}^2  \\
&\leq \frac{1}{4\epsilon_1}\norm{h_1^k}{}^2 + \epsilon_1\norm{\theta_1^k}{}^2 + \frac{1}{4\epsilon_2}\norm{h_2^k}{}^2 + \epsilon_2\norm{\theta_2^k}{}^2 + \frac{b_1-b_2}{2}\norm{\theta_2^k}{}^2,
\end{align*}
where we used Young's inequality with $\epsilon_1$ and $\epsilon_2$. This leads to
\begin{align*}
&\frac{1}{2\tau}\left(\norm{\theta_1^k}{}^2 - \norm{\theta_1^{k-1}}{}^2 +  \norm{\theta_2^k}{}^2 - \norm{\theta_2^{k-1}}{}^2  + \norm{\theta_1^k-\theta_1^{k-1}}{}^2  + \norm{\theta_2^k-\theta_2^{k-1}}{}^2\right)\\
&\quad + \kappa_1\norm{\grad \theta_1^k}{}^2 + (c_1-\epsilon_1)\norm{\theta_1^k}{}^2   + \kappa_2\norm{\grad \theta_2^k}{}^2 + \left(c_2-\frac{b_1-b_2}{2}-\epsilon_2\right)\norm{\theta_2^k}{}^2  \\
&\leq \frac{1}{4\epsilon_1}\norm{h_1^k}{}^2 + \frac{1}{4\epsilon_2}\norm{h_2^k}{}^2.
\end{align*}
Summing up from $k=1$ to $N$ and taking $\epsilon_1, \epsilon_2$ sufficiently small, we obtain,
\begin{align*}
&\frac{1}{2\tau}\left(\norm{\theta_1^N}{}^2 +  \norm{\theta_2^N}{}^2 + \sum_{k=1}^N\norm{\theta_1^k-\theta_1^{k-1}}{}^2  + \sum_{k=1}^N\norm{\theta_2^k-\theta_2^{k-1}}{}^2\right)\\
&\quad + \sum_{k=1}^N\left(\kappa_1\norm{\grad \theta_1^k}{}^2  + \kappa_2\norm{\grad \theta_2^k}{}^2 + C_1\norm{\theta_1^k}{}^2 + C_2\norm{\theta_2^k}{}^2\right)\\
&\leq \frac{1}{2\tau}\left(\norm{\theta_{10}}{}^2  + \norm{\theta_{20}}{}^2\right)  + \frac{1}{4\epsilon_1}\sum_{k=1}^N\norm{h_1^k}{}^2 + \frac{1}{4\epsilon_2}\sum_{k=1}^N\norm{h_2^k}{}^2,
\end{align*}
giving upon multiplying by $\tau$ the following intermediary uniform (in $N$ and $k$) bounds:
\begin{equation}\label{eq:pre5}
\norm{\theta_i^k}{} \leq C\qquad\text{and}\qquad \tau\sum_{k=1}^N\norm{\theta_i^k}{H^1(\Omega)}^2 \leq C.
\end{equation}
For the bound on the difference quotients, test the $\theta_i^k$ equation with $\theta_i^k - \theta_i^{k-1}$, and defining $L_i := -\kappa_i\Delta + c_i$ and using Young's inequality with $\tau\slash \epsilon_i$, we obtain
\begin{align*}
\frac{1}{\tau}\norm{\theta_i^k-\theta_i^{k-1}}{}^2 + \langle L_i\theta_i^k, \theta_i^k-\theta_i^{-1} \rangle &= \int_\Omega h_i^k(\theta_i^k-\theta_i^{k-1}) + (-1)^ib_i(\theta_1^k-\theta_2^k)\chi_k(\theta_i^k-\theta_i^{k-1})\\
&\leq \frac{\tau}{4\epsilon_0}\norm{h_i^k}{}^2 + \frac{\epsilon_0}{\tau}\norm{\theta_i^k-\theta_i^{k-1}}{}^2 + \frac{b_ih}{4\epsilon_1}\norm{\theta_1^k}{}^2 + \frac{b_i\epsilon_1}{\tau}\norm{\theta_i^k-\theta_i^{k-1}}{}^2\\
&\quad + \frac{b_ih}{4\epsilon_2}\norm{\theta_2^k}{}^2 + \frac{b_i\epsilon_2}{\tau}\norm{\theta_i^k-\theta_i^{k-1}}{}^2.
\end{align*}
 Here we use the linearity and self-adjointedness of $L_i$ to derive
\begin{align*}
\langle L_i\theta_i^k, \theta_i^k - \theta_i^{k-1} \rangle  &= \frac 12 \langle L_i\theta_i^k-L_i\theta_i^{k-1}, z_n - \theta_i^{k-1}\rangle  + \frac 12 \langle L_i\theta_i^k-L_i\theta_i^{k-1}, \theta_i^k - \theta_i^{k-1} \rangle + \langle L_i\theta_i^{k-1}, \theta_i^k - \theta_i^{k-1} \rangle\\
&= \frac 12 \langle L_i\theta_i^k-L_i\theta_i^{k-1}, \theta_i^k - \theta_i^{k-1}\rangle  + \frac 12 \langle L_i\theta_i^k, \theta_i^k \rangle + \frac 12 \langle L_i\theta_i^{k-1}, \theta_i^{k-1} \rangle - \langle A\theta_i^k, \theta_i^{k-1} \rangle + \langle L_i\theta_i^{k-1}, \theta_i^k  \rangle\\
&\quad - \langle A\theta_i^{k-1},  \theta_i^{k-1} \rangle\\
&= \frac 12 \langle L_i\theta_i^k-L_i\theta_i^{k-1}, \theta_i^k - \theta_i^{k-1}\rangle  + \frac 12 \langle L_i\theta_i^k, \theta_i^k \rangle - \frac 12 \langle L_i\theta_i^{k-1}, \theta_i^{k-1} \rangle \\
&\geq \frac12 \min(\kappa_i, c_i)\norm{\theta_i^k-\theta_i^{k-1}}{H^1(\Omega)}^2 + \frac 12 \langle L_i\theta_i^k, \theta_i^k \rangle - \frac 12 \langle L_i\theta_i^{k-1}, \theta_i^{k-1} \rangle.
\end{align*}
Thus
\begin{align*}
&\frac 12 \langle L_i\theta_i^k, \theta_i^k \rangle - \frac 12 \langle L_i\theta_i^{k-1}, \theta_i^{k-1} \rangle+\left(\frac{1}{\tau}-\frac{\epsilon_0 + b_i\epsilon_1 + b_i\epsilon_2}{\tau}\right)\norm{\theta_i^k-\theta_i^{k-1}}{}^2 + \frac12 \min(\kappa_i, c_i)\norm{\theta_i^k-\theta_i^{k-1}}{H^1(\Omega)}^2\\
 &\leq \frac{\tau}{4\epsilon_0}\norm{h_i^k}{}^2 + \frac{b_ih}{4\epsilon_1}\norm{\theta_1^k}{}^2 +  \frac{b_ih}{4\epsilon_2}\norm{\theta_2^k}{}^2.
\end{align*}
Summing and using coercivity of $L_i$, we get
\begin{align*}
\nonumber &C_1\norm{\theta_i^N}{H^1(\Omega)}^2 + \frac{C_2}{\tau}\sum_{k=1}^N\norm{\theta_i^k-\theta_i^{k-1}}{}^2 + C_3\sum_{k=1}^N\norm{\theta_i^k-\theta_i^{k-1}}{H^1(\Omega)}^2 \\
\nonumber  &\leq \frac{\tau}{4\epsilon_0}\sum_{k=1}^N\norm{h_i^k}{}^2 + \frac{b_ih}{4\epsilon_1}\sum_{k=1}^N\norm{\theta_1^k}{}^2 +  \frac{b_ih}{4\epsilon_2}\sum_{k=1}^N\norm{\theta_2^k}{}^2 +  \frac 12 \langle L_i\theta_{i0}, \theta_{i0} \rangle\\
 &\leq C
\end{align*}
using \eqref{eq:pre5}. This gives the first and last bound stated in the lemma. Now, rearrange the $\theta_i^N$ equation and take the $L^2(\Omega)$ norm to find 
\[\kappa_i \norm{\Delta \theta_i^k}{L^2(\Omega)} \leq \norm{h_i^k}{L^2(\Omega)} +\tau^{-1}\norm{\theta_i^{k-1}-\theta_i^k}{L^2(\Omega)}+ c_i\norm{\theta_i^k}{L^2(\Omega)} + b_i\norm{\theta_1^k-\theta_2^k}{L^2(\Omega)},\]
which upon squaring and multiplying by $\tau$ leads to
\begin{align*}
C\tau\norm{\Delta \theta_i^k}{L^2(\Omega)}^2 \leq \tau\norm{h_i^k}{L^2(\Omega)}^2 +\frac{\norm{\theta_i^{k-1}-\theta_i^k}{L^2(\Omega)}^2}{\tau} + c_i^2\tau\norm{\theta_i^k}{L^2(\Omega)}^2 + b_i^2\tau\norm{\theta_1^k-\theta_2^k}{L^2(\Omega)}^2.
\end{align*}
Summing and using the previous bound, we get the second bound stated in the lemma.
\end{proof}
As a result, we obtain the following bounds for the interpolants constructed from $\theta_i^k$.
\begin{lem}\label{lem:thetaINBounds}
If $h_1, h_2 \in L^\infty(0,T;H^1(\Omega)^*)$, then the following bound holds uniformly in $N$:
\begin{align*}
&\norm{\theta_i^N}{L^\infty(0,T;H^1(\Omega))}+ \norm{\Delta \theta_i^N}{L^2(0,T;L^2_{\loc}(\Omega))}  + \norm{\hat \theta_i^N}{L^\infty(0,T;H^1(\Omega))} + \norm{\Delta \hat \theta_i^N}{L^2(0,T;L^2_{\loc}(\Omega))}\\
&\quad+ \norm{\partial_t \hat \theta_i^N}{L^2(0,T;L^2(\Omega)) \cap L^\infty(0,T;H^1(\Omega)^*)} \leq C.
\end{align*}
\end{lem}
\begin{proof}
This is a consequence of Lemma \ref{lem:discBounds}. 
The bound in $L^\infty$ in time on $\hat\theta_i^N$ is obtained due to the calculation
\begin{align*}
\norm{\hat \theta_i^N(t)}{H^1(\Omega)} &\leq \sum_{k=1}^N \norm{\theta_i^{k-1}}{H^1(\Omega)}\chi_{I_k}(t) + \left\lVert\sum_{k=1}^N \frac{\theta_i^k-\theta_i^{k-1}}{\tau}(t-t_{k-1})\chi_{I_k}(t)\right\rVert_{H^1(\Omega)}\\
&\leq C\sum_{k=1}^N \chi_{I_k}(t) + \sum_{k=1}^N \chi_{I_k}(t)\frac{(t-t_{k-1})}{\tau} \left\lVert\theta_i^k-\theta_i^{k-1}\right\rVert_{H^1(\Omega)}\tag{using the first bound of Lemma \ref{lem:discBounds} }\\
&\leq C+ 2C\sum_{k=1}^N \chi_{I_k}(t)\tag{as above and estimating $t-t_{k-1} \leq \tau$ on $I_k$}\\
&= 3C,
\end{align*}
 and the bound on its time derivative follows simply by rearranging the equation for $\theta_i^N$.
\end{proof}

 If $g \in L^\infty(0,T;L^2(\Omega))$, we clearly have that the right-hand side of the equation for $\Phi^N$ is bounded in $L^\infty(0,T;L^2(\Omega))$ and therefore
\[\norm{\Phi^N}{L^\infty(0,T;H^1(\Omega))} + \norm{\Delta \Phi^N}{L^\infty(0,T;L^2(\Omega))}   \leq C\quad\text{uniformly in $N$}.\]
Testing the inequality for $u^N$ with $\Phi^N$, manipulating with Young's inequality with epsilon, Poincar\'e's inequality and using the bound on $\Phi^N$, we also easily derive 
\[\norm{u^N}{L^\infty(0,T;H^1_0(\Omega))} \leq C\quad\text{uniformly in $N$}.\]
The absence of $L^\infty(0,T;H^2_{\loc}(\Omega))$ regularity for $\theta_i^N$ has the effect that we do not get an $L^\infty(0,T;H^2_{\loc}(\Omega))$ bound for $u^N$ for the general coefficient functions $a$.
\begin{lem}\label{lem:forAConst}Suppose $f \in L^\infty(0,T;L^2(\Omega))$ and that the coefficient function $a$ is constant. Then
\[\norm{\Delta u^N}{L^\infty(0,T;L^2(\Omega))}    \leq C\quad\text{uniformly in $N$}.\]
\end{lem}
\begin{proof}
Since $\min(x,y) = x-(x-y)^+$, the Lewy--Stampacchia inequality from Proposition \ref{prop:ls} for the discretised solutions reads $- (f^k+a\Delta \Phi^k)^+ \leq  -a\Delta u^k  -f^k \leq 0$, so if we multiply by $\chi_{I_k}$ and sum up, we obtain
\[- (f^N+a\Delta \Phi^N)^+ \leq  -a\Delta u^N  -f^N \leq 0\qquad \text{a.e. in $\Omega$}.\]
Thanks to the $L^\infty(0,T;L^2(\Omega))$ bound on $\Delta \Phi^N$, $-\Delta u^N\in L^\infty(0,T;L^2(\Omega))$ uniformly.
\end{proof}
If $a$ were not restricted to being a constant, we can obtain $-\Delta u^N \in L^\infty(0,T;L^1(\Omega))$ uniformly but the lack of $L^p$ elliptic regularity for $p=1$ means we cannot proceed any further with just this information to get local $W^{1,p}$ regularity.
\subsection{Limiting behaviour}\label{sec:limitingBehaviour}
Putting together the previous boundedness results in Bochner spaces and using interior elliptic regularity, we get the existence of $(\theta_1,\theta_2,\tilde \chi)$ such that the following convergences hold for subsequences that we have relabelled:
\begin{equation}\label{eq:listOfBochnerConv}
\begin{aligned}
\theta_i^N &\weaklystar \theta_i &&\text{in $L^\infty(0,T;H^1(\Omega))$ and weak in $L^2(0,T;H^2_{\loc}(\Omega))$},\\
\hat \theta_i^N &\weaklystar \theta_i &&\text{in $L^\infty(0,T;H^1(\Omega))$ and weak in $L^2(0,T;H^2_{\loc}(\Omega))$},\\
\partial_t \hat \theta_i^N &\weaklystar \partial_t\theta_i &&\text{in $L^\infty(0,T;H^1(\Omega)^*)$ and weak in $L^2(0,T;L^2(\Omega))$},\\
\chi^N &\weaklystar \tilde \chi &&\text{in $L^\infty(Q)$}.\\
\end{aligned}
\end{equation}
 Let us make a note that under \eqref{ass:parabolicAssConsts}, we obtain
$\theta_1^N \geq \theta_2^N \text{ and (hence) } \theta_1 \geq \theta_2.$ 
That the weak-* limits of $\theta_i^N$ and $\hat \theta_i^N$ are the same is proved in Lemma \ref{lem:differentInterpolants} in the appendix, in which one observes also that
\begin{equation}\label{eq:limitsOfhatandNormal}
\theta_i^N - \hat \theta_i^N \to 0 \quad \text{in $L^2(0,T;L^2(\Omega))$}.
\end{equation}
Applying the Aubin--Lions theorem, we further obtain
\begin{align*}
\hat \theta_i^N &\to \theta_i \quad\text{in $L^2(0,T;X) \cap C^0([0,T];Y)$}
\end{align*}
where $X$ and $Y$ are Banach spaces such that $H^2_{\loc}(\Omega) \ctsCompact X \cts L^2_{\loc}(\Omega)$ and $H^1(\Omega) \ctsCompact Y \cts L^2(\Omega).$  In particular, using \eqref{eq:limitsOfhatandNormal},
\begin{align*}
\theta_i^N \to \theta_i \quad\text{in $L^2(0,T;L^2(\Omega))$.}
\end{align*}
We can pass to the limit in the sense of Bochner in the $\theta_i^N$ equations to obtain
\begin{equation}\label{eq:bochnerEquations}
\begin{aligned}
\int_0^T\int_\Omega \partial_t \theta_i \eta + \kappa_i\grad \theta_i \cdot \grad \eta + c_i\theta_i\eta &= \int_0^T\int_\Omega(h_i + (-1)^ib_i(\theta_1-\theta_2)\tilde \chi)\eta &&\forall \eta \in L^2(0,T;H^1(\Omega)),\\
\theta_i(0) &= \theta_{i0} &&\text{in $\Omega$},
\end{aligned}
\end{equation}
where we used the convergence in $C^0([0,T];L^2(\Omega))$ to recover the initial condition.

It would also be possible to pass to the Bochner limit in the equation for $\Phi^N$ but \textit{not} for the quasi-variational inequality for $u^N$. This is because we would need a strong Bochner convergence for either $u^N$ or $\Phi^N$ in $L^2(0,T;H^1_0(\Omega))$ to take the limit after testing the inequality for $u^N$ with an appropriate recovery sequence, just as in the proof of Proposition \ref{prop:existenceForRegularisedProblem}. In order to brute force such a strong convergence, we have to use the $L^\infty$ in time bounds and work on the level of pointwise a.e. fixed times.  Indeed, using the above uniform bounds, we get the existence of a subsequence $N_j \equiv N_j(t)$ and limiting functions $\Phi(t)$, $u(t)$ and $\chi(t)$ such that the following convergences hold:
\begin{equation}\label{eq:pointwiseinTimeConv}
\begin{aligned}
\theta_i^{N}(t) &\weaklyto \theta_i(t)&&\text{in $H^1(\Omega)$},\\
\Phi^{N_j}(t) &\weaklyto \Phi(t) &&\text{in $H^2_{\loc}(\Omega)$ and strong in $H^1_0(\Omega)$},\\
\Delta \Phi^{N_j}(t) &\weaklyto \Delta \Phi(t) &&\text{in $L^2(\Omega)$},\\
u^{N_j}(t) &\weaklyto u(t) &&\text{in $H^1_0(\Omega)$},\\
\chi^{N_j}(t) &\weaklystar \chi(t)&&\text{in $L^\infty(\Omega)$},
\end{aligned}
\end{equation}
where we used the boundedness of $\Delta\Phi^N(t)$ in $L^2(\Omega)$, the compact embedding $L^2(\Omega) \ctsCompact H^{-1}(\Omega)$ and the coercivity of the Laplacian to obtain the strong $H^1_0(\Omega)$ convergence above. Note also that we have the first convergence listed above for the whole sequence $\{\theta_i^N(t)\}$ because we already know that $\theta_i^N \to \theta_i$ in $L^2(0,T;L^2(\Omega))$. At present, we cannot claim that $\Phi$ and $u$ constructed above are the same as the weak Bochner limits of the sequences $\Phi^N$ and $u^N$ respectively.

\begin{remark}
It is important to keep in mind that the convergences in \eqref{eq:pointwiseinTimeConv} hold for a subsequence that itself depends on the time point, which is a major issue. In other works addressing quasistatic contact problems, this type of dependence is circumvented  and a global subsequence can be found. The idea there (see for example \cite{Andersson1991, Cocu1996}) is to use a diagonalisation argument and an assumption on the differentiability of the source term to obtain an estimate on the difference $u^{k+1}-u^k$. The term involving $\chi^N$ that appears in the equations for $\theta_i^N$ and $\Phi^N$ impedes the derivation of such an estimate. We will resolve this issue by proving directly the convergence for the entire sequence $\{\chi^N(t)\}$ under the strong non-degeneracy assumption \eqref{ass:strongNonDegeneracy}.
\end{remark}

It is easy to pass to the limit in the equation for $\Phi$ in \eqref{eq:parabolicNEquations} in a pointwise a.e. sense. For the inequality for $u^N$, one can take an arbitrary function $v(t) \in H^1_0(\Omega)$ with $v(t) \leq \Phi(t)$, define $v^N(t):=v(t)-\Phi(t) + \Phi^{N}(t) = \sum_{k=1}^N (v(t)-\Phi(t) + \Phi^k)\chi_{I_k}(t)$, test the variational inequality with $v^{N_j}(t)$ and then pass to the limit with the aid of the strong $H^1$ convergence for $v^N(t)$ and Minty's lemma (just as in the proof of Proposition \ref{prop:existenceForRegularisedProblem}). 
We end up with
\begin{equation}\label{eq:pointwiseEquations}	
\begin{aligned}
\int_\Omega \grad \Phi(t)\cdot \grad \xi &= \int_\Omega (\alpha(\theta_1(t)-\theta_2(t))\chi(t) + g(t))\xi &&\forall \xi \in H^1_0(\Omega),\\
u(t) \in \mathbb{K}(\Phi(t)) : \int_\Omega a(\theta_1(t))\grad u(t) \cdot \grad (u(t)-v(t)) &\leq \int_\Omega f(t)(u(t)-v(t)) &&\forall v(t) \in \mathbb{K}(\Phi(t)),
\end{aligned}
\end{equation}
for almost every $t \in (0,T).$ Now, let us give a first characterisation of $\chi(t)$.
\begin{prop}
We have $\chi(t) \in 1-H(\Phi(t)-u(t))$  for almost all $t \in (0,T)$.
\end{prop}
\begin{proof}
From \eqref{eq:usefulId}, we know that $\int_\Omega (\Phi^k-u^k)^+\chi_k = 0.$ Multiplying this by $\chi_{I_{kN_j}}(t)$ where $I_{kN_j}$ stands for the $k$th interval of the partition of $[0,T]$ into $N_j$ subintervals,  summing, and then passing to the limit, we find
\[
\int_\Omega (\Phi(t)-u(t))^+\chi(t)=0,\]
which tells us that when $u(t) < \Phi(t)$, we must have $\chi(t) = 0$, i.e., $\chi(t) \leq \chi_{\{\Phi(t)=u(t)\}}.$
\end{proof}
Note that the equations for $\theta_i$ in \eqref{eq:bochnerEquations} are, as currently written, uncoupled to the equation and inequality in \eqref{eq:pointwiseEquations} because $\tilde \chi$ is not necessarily equal to $\chi$. To make such an identification, it appears that the identification of $\chi$ as the characteristic function is necessary. The next section is devoted to this.
\subsection{Identification of the characteristic function}\label{sec:parabolicIdChar}
We now enforce the regularity stated in Theorem \ref{thm:existenceParabolic} on the source terms and initial data as well taking the coefficient function $a$ to be a constant. Due to Lemma \ref{lem:forAConst}, in addition to the convergences listed before, we also get
\begin{equation}\label{eq:uNJconvergence}
u^{N_j}(t) \weaklyto u(t) \quad \text{in $H^2_{\loc}(\Omega)$ and strong in $H^1_0(\Omega)$}.
\end{equation}
In the next proposition, we will identify $\chi$ as the expected characteristic function. The assumption we need for it is slightly weaker than the one made in the theorem (assumption \eqref{ass:strongNonDegeneracy}) which we shall need later to couple \eqref{eq:bochnerEquations} and \eqref{eq:pointwiseEquations}.

\begin{prop}[\textsc{Identification of $\chi$ as the characteristic function}]\label{prop:idCharPar}
Let $a' \equiv 0$ and let \eqref{ass:onConstants}, \eqref{ass:parabolicAssConsts} and \eqref{ass:conditionToGetCharPar} hold. 
 Then for almost all $t \in (0,T)$,
\[\chi(t) = \chi_{\{\Phi(t)=u(t)\}}.\]
\end{prop}
\begin{proof}

By Lemma \ref{lem:conditionForNonDegSemiDisc}, Proposition \ref{prop:regularitySemiDisc} is in force and $\chi^k$ can be identified as the characteristic function $\chi_{\{\Phi^k = u^k\}}$. 
Hence\footnote{This is easy to see: the left-hand side is non-zero and equal to $1$ if and only if there exists a $j \in \{1, ..., N\}$ such that $t \in I_j$ and $x \in \{\Phi^j=u^j\}$. The right-hand side is non-zero and equal to $1$ if and only if $x,t$ are such that $\Phi^N(t,x)-u^N(t,x) = 0 \iff \sum_{k=1}^N (\Phi^k(x)-u^k(x))\chi_{I_k}(t)=0$, i.e., if and only if $t \in I_j$ and $x \in \{\Phi^j = u^j\}$ for some $j$.}
\begin{equation}\label{eq:chiNEqualsCharacteristic}
\chi^N(t,x) = \sum_{k=1}^N \chi_{\{\Phi^k = u^k\}}(x)\chi_{I_k}(t) 
= \chi_{\{\Phi^N(t) = u^N(t)\}}(x).
\end{equation}
Because $u(t)$ and $\Phi(t)$ are both in $H^2_{\loc}(\Omega)$, arguing like in the proof of Theorem \ref{thm:existenceRegularity} in \S \ref{sec:ellipticIdChar}, we deduce that by virtue of $u(t)$ satisfying the variational inequality in \eqref{eq:pointwiseEquations}, for every $\Omega' \subset\subset \Omega$,
\begin{equation}\label{eq:qLimiting}
-a\Delta u(t) + (f(t)+a\Delta \Phi(t))\chi_{\{\Phi(t) = u(t)\}} = f(t) \quad \text{on $\Omega'$}.
\end{equation}
On the other hand, the quasi-variational inequality for $u^k$ in \eqref{eq:parabolicEquationDiscWeaker} can be written as
\[-a\Delta u^k + (f^k+a\Delta \Phi^k)\chi_{\{\Phi^k = u^k\}} = f^k \quad \text{on $\Omega'$},\]
whence
\begin{equation}\label{eq:qN}
-a\Delta u^N(t) + (f^N(t)+a\Delta \Phi^N(t))\chi^N(t) = f^N(t) \quad \text{on $\Omega'$}.
\end{equation}
We will pass to the limit in this equation for the subsequence $N_j$, which for simplicity we shall denote by $N$ from now on. Making use of the weak convergences listed in \eqref{eq:pointwiseinTimeConv} and \eqref{eq:uNJconvergence}, 
$-a\Delta u^N(t) -f^N(t) \weaklyto -a\Delta u(t) -f(t)$ in $L^2_{\loc}(\Omega)$ and we get from \eqref{eq:qLimiting} and \eqref{eq:qN}
\begin{equation}\label{eq:toCompare2}
(f^N(t)+a\Delta \Phi^N(t))\chi^N(t) \weaklyto (f(t)+a\Delta \Phi(t))\chi_{\{\Phi(t)=u(t)\}} \quad\text{in $L^2_{\loc}(\Omega)$}.
\end{equation}
Since $(\chi^N(t))^2 \leq 1$, it has a weak-* limit $\zeta(t)$ for a subsequence (which we have relabelled). We use this fact on the left-hand side of the above, expanding it as 
\begin{align*}
(f^N(t)+a\Delta\Phi^N(t))\chi^N(t)&=(f^N(t)  - ag^N(t)- \alpha a(\theta_1^N(t)-\theta_2^N(t))\chi^N(t))\chi^N(t)\\
&=(f^N(t) - ag^N(t))\chi^N(t) - \alpha a(\theta_1^N(t))(\theta_1^N(t)-\theta_2^N(t))(\chi^N(t))^2\\
&\weaklyto (f(t)-ag(t))\chi(t) -\alpha a(\theta_1(t)-\theta_2(t))\zeta(t)\\
&\leq (f(t)-ag(t))\chi(t) -\alpha a(\theta_1(t)-\theta_2(t))(\chi(t))^2\\
& = (f(t)+ a\Delta \Phi(t))\chi(t)
\end{align*}
where we again used \cite[Lemma 2]{MR584398} for the penultimate line because $\theta_1 \geq \theta_2$ (as we explained at the start of \S \ref{sec:limitingBehaviour}).
 Comparing this with \eqref{eq:toCompare2}, we get
\begin{equation}\label{eq:100}
(f(t)+a\Delta \Phi(t))(\chi_{\{\Phi(t)=u(t)\}} - \chi(t)) \leq 0.
\end{equation}
Now, the non-degeneracy estimate for each $k$ we derived in \eqref{eq:estimateForDegeneracyInK} implies that
\[f^N(t) + a\Delta \Phi^N(t) \geq \mu > 0,\]
whence taking the subsequence $N_j(t)$, using the weak convergence of $\Delta \Phi^{N_j}(t)$ and passing to the limit, we obtain
\[f(t) + a\Delta \Phi(t) \geq \mu > 0.\]
This implies from \eqref{eq:100} that $\chi_{\{\Phi(t)=u(t)\}} \leq \chi(t)$ and therefore $\chi(t)=\chi_{\{\Phi(t)=u(t)\}} .$ 
\end{proof}
The subsequence principle cannot be used to say that the entire sequence $\{\chi^N(t)\}$  converges to $\chi_{\{\Phi(t)=u(t)\}}$ since this object depends explicitly on $\Phi(t)$ and $u(t)$, which depend on the subsequence $N_j(t)$ that is taken. That is, the limit is not necessarily unique. In the next section, we will show that the dependence on time in the subsequence can be removed.
\subsection{Identification of Bochner and pointwise limits in time}\label{sec:concExistence}
In order to equate $\tilde\chi$ and $\chi$ and thereby fully couple \eqref{eq:bochnerEquations} and \eqref{eq:pointwiseEquations}, we need the additional non-degeneracy assumption \eqref{ass:strongNonDegeneracy} stated in the theorem. Before we proceed, observe that by multiplying \eqref{eq:boundOnTheta1K} by $\chi_{I_k}(t)$ and summing, we obtain
\begin{align}
\norm{\theta_1^N(t)}{L^\infty(\Omega)} &= \norm{\sum_{k=1}^N \theta_1^k \chi_{I_k}(t)}{L^\infty(\Omega)} = \sum_{k=1}^N \norm{\theta_1^k}{L^\infty(\Omega)} \chi_{I_k}(t) \leq \norm{h_1}{L^1(0,T;L^\infty(\Omega))} + \norm{\theta_{10}}{L^\infty(\Omega)}.\label{eq:LinftyTheta1NT}
\end{align}
\begin{prop}[\textsc{Convergence of $\{\chi^N(t)\}$ to $\chi(t)$}]\label{prop:new1}
Under the assumptions of Proposition \ref{prop:idCharPar}, if \eqref{ass:strongNonDegeneracy} holds,
then for almost every $t \in (0,T)$,
\[\chi^N(t) \to \chi(t) \text{ in $L^p(\Omega)$ for all $p < \infty$}.\]
\end{prop}
\begin{proof}
Assumption \eqref{ass:strongNonDegeneracy} implies the existence of a constant $\mu > 0$ such that
\begin{align}
f-ah - \alpha a\norm{h_1}{L^1(0,T;L^\infty(\Omega))} - \alpha a\norm{\theta_{10}}{L^\infty(\Omega)} &\geq \mu > a\alpha \norm{h_1}{L^1(0,T;L^\infty(\Omega))} + a\alpha\norm{\theta_{10}}{L^\infty(\Omega)}\label{eq:strongNonDegeneracyImplication}.
\end{align}
As in the proof of Proposition \ref{prop:idCharPar}, we can derive 
from \eqref{eq:estimateForDegeneracyInK} the inequalities
\begin{align*}
f^N(t) + a\Delta \Phi^N(t) &\geq \mu > 0,\\
f(t) + a\Delta \Phi(t) &\geq \mu > 0.
\end{align*}
These two lower bounds are non-degeneracy conditions and they allow us to apply the $L^1$ continuous dependence estimate for characteristic functions in \cite[Theorem 4.7, \S 5:4]{Rodrigues} for the two obstacle problems satisfied by $u(t)$ and $u^N(t)$ and we get, making use of the equations for  $\Phi^N(t)$ and $\Phi(t)$, for almost every $t \in (0,T)$, 
\begin{align}
\nonumber \mu\norm{\chi_{\{\Phi(t)=u(t)\}}-\chi_{\{\Phi^N(t)=u^N(t)\}}}{L^1(\Omega)} 
\nonumber &\leq \norm{f(t)-f^N(t)}{L^1(\Omega)} + \norm{a\Delta \Phi(t) - a\Delta \Phi^N(t)}{L^1(\Omega)}\\
\nonumber &\leq \norm{f(t)-f^N(t)}{L^1(\Omega)} + a\norm{g(t)-g^N(t)}{L^1(\Omega)}\\
\nonumber &\quad + a\alpha\norm{(\theta_1(t)-\theta_2(t))\chi(t)-(\theta_1^N(t)-\theta_2^N(t))\chi^N(t)}{L^1(\Omega)}\\
\nonumber &\leq \norm{f(t)-f^N(t)}{L^1(\Omega)} + a\norm{g(t)-g^N(t)}{L^1(\Omega)}\\
\nonumber &\quad + a\alpha \norm{(\theta_1(t)-\theta_2(t)-(\theta_1^N(t)-\theta_2^N(t)))\chi(t)}{L^1(\Omega)}\\
&\quad + a\alpha\norm{(\theta_1^N(t)-\theta_2^N(t))(\chi(t)-\chi^N(t))}{L^1(\Omega)}\label{eq:new1eq}\\
\nonumber &\leq \norm{f(t)-f^N(t)}{L^1(\Omega)} + a\norm{g(t)-g^N(t)}{L^1(\Omega)}\\
\nonumber &\quad + a\alpha\norm{(\theta_1(t)-\theta_2(t))-(\theta_1^N(t)-\theta_2^N(t))}{L^1(\Omega)}\\
&\nonumber \quad+ a\alpha \hat L\norm{\chi(t)-\chi^N(t)}{L^1(\Omega)} ,
\end{align}
where for the final line we used
\[\norm{\theta_1^N(t)-\theta_2^N(t)}{L^\infty(\Omega)} \leq \norm{\theta_1^N(t)}{L^\infty(\Omega)}\leq \norm{h_1}{L^1(0,T;L^\infty(\Omega))} + \norm{\theta_{10}}{L^\infty(\Omega)} =:\hat L\] due to the non-negativity of $\theta_2^N$ and by the $L^\infty$ estimate \eqref{eq:LinftyTheta1NT}. 
Taking into account that $\chi(t)=\chi_{\{\Phi(t)=u(t)\}}$ and $\chi^N(t) = \chi_{\{\Phi^N(t)=u^N(t)\}}$ (see \eqref{eq:chiNEqualsCharacteristic}), the above becomes
\begin{align*}
(\mu-a\alpha \hat L)\norm{\chi_{\{\Phi(t)=u(t)\}}-\chi_{\{\Phi^N(t)=u^N(t)\}}}{L^1(\Omega)}
&\leq \norm{f(t)-f^N(t)}{L^1(\Omega)} + a\norm{g(t)-g^N(t)}{L^1(\Omega)}\\
&\quad + a\alpha\norm{(\theta_1(t)-\theta_2(t))-(\theta_1^N(t)-\theta_2^N(t))}{L^1(\Omega)}.
\end{align*}
Using then \eqref{eq:strongNonDegeneracyImplication} and the definition of $\hat L$, the coefficient on the left-hand side above is positive. 
Then taking $N \to \infty$, using the fact that $f^N \to f$, $g^N \to g$ in $L^2(0,T;L^2(\Omega))$ \cite[Remark 8.15]{Roubicek} and $\theta_i^N \to \theta_i$ in $L^2(0,T;L^2(\Omega))$, we get that for almost every $t \in (0,T)$,
\[\chi^N(t) \to \chi_{\{u(t)=\Phi(t)\}} = \chi(t) \quad \text{in $L^1(\Omega)$}\]
at least for a subsequence (independent of $t$) which we have relabelled. The convergence is also in $L^p(\Omega)$ for any $p < \infty$ because the sequence and its limit are characteristic functions.
\end{proof}
This strong convergence for the global (independent of $t$) sequence implies via the Lebesgue's dominated convergence theorem (and \eqref{eq:listOfBochnerConv}) that \[\chi^N \to \chi \quad\text{ in $L^p(Q)$ for all $p<\infty$ and weak-* in $L^\infty(Q)$}\]
and crucially, from \eqref{eq:listOfBochnerConv},
\[\tilde \chi \equiv \chi.\]
This is sufficient to conclude existence because now the Bochner limiting equations in \eqref{eq:bochnerEquations}  for $\theta_i$ and the pointwise a.e. in time equation and inequality for $\Phi$ and $u$ respectively in \eqref{eq:pointwiseEquations} are coupled through the single term $\chi$.

Let us finish by discussing membership in the Bochner classes that we claimed in the statement of Theorem \ref{thm:existenceParabolic}. For a.e. $t \in (0,T)$, 
\begin{align*}
-\Delta \Phi^N(t) + \Delta\Phi(t) = \alpha(\theta_1^N(t)-\theta_2^N(t) - (\theta_1(t) - \theta_2(t)))\chi^N(t) + (\theta_1(t)-\theta_2(t))(\chi^N(t)-\chi(t)) + g^N(t)-g(t),
\end{align*}
and we see that the first and third term trivially converge to zero in $L^2(\Omega)$, as does the second term since we have $\chi^N(t) \to \chi(t)$ in $L^{2n\slash(2+n)}(\Omega)$ and the factor $\theta_1(t)-\theta_2(t) \in L^{2n\slash (n-2)}(\Omega)$ (this is due to the embedding $H^1(\Omega) \cts L^{2^*}(\Omega)$ where $2^* = 2n\slash (n-2)$ is the Sobolev conjugate). From this, we obtain
\begin{align*}
\Phi^N(t) &\to \Phi(t) \quad\text{ in $H^2_{\loc}(\Omega)$},\\
u^N(t) &\to u(t) \quad\text{ in $H^1_0(\Omega)$},
\end{align*}
for the full sequence. From this and the bounds on $\Phi^N$ and $u^N$ in \S \ref{sec:interpolants}, we may apply the Lebesgue's dominated convergence theorem which implies convergence in Bochner spaces and we get Bochner measurability of the limiting functions since the pointwise limit of measurable functions is also measurable. This concludes the proof of Theorem \ref{thm:existenceParabolic}.
\begin{remark}\label{rem:onStrongNonDegAss}
The assumption \eqref{ass:strongNonDegeneracy} and its weaker version \eqref{ass:conditionToGetCharPar} are sufficient conditions and are specially formulated around the $L^\infty$ result for $\theta_1$ via Lemma \ref{lem:LinftyBoundTheta1}; there are different assumptions one could make instead. For example in low dimensions using the embedding of $H^2(\Omega) \cts L^\infty(\Omega)$  we can get a different bound on the $L^\infty$ norm of $\theta_1$ under some elliptic regularity assumptions, but this would now depend on both $h_1$ and $h_2$ (and hence $\theta_2$ as well), unlike the current hypotheses which depend only on $h_1$. One could also attempt to estimate the $L^\infty$ norm of $\theta_1-\theta_2$ (instead of simply neglecting $\theta_2$ like we do in the proof of Proposition \ref{prop:new1}, see \eqref{eq:new1eq}) by using e.g. Remark \ref{rem:linftyBoundOnDifferenceThetas}.
\end{remark}
\subsection{Continuity in time of solutions}\label{sec:parCont}
The following proposition serves to prove the principal claim of Theorem \ref{thm:continuityParabolic}.

\begin{prop}Let $f, g \in C^{0,\gamma}((0,T);L^1(\Omega))$ for some $\gamma \in (0,1]$ and let \eqref{ass:strongNonDegeneracy} hold.
Then for every $t \in (0,T)$ and all $p < \infty$, $\chi^N(t) \to \chi(t)$ in $L^p(\Omega)$ and $\chi \in C^0((0,T);L^p(\Omega))$.
\end{prop}
\begin{proof}
Recall that we used piecewise constant (in time) interpolants based on a partition of $[0,T]$ into disjoint intervals and we approximated the source terms also in a piecewise constant fashion. This means then that the equation for $\Phi^N$ and the inequality for $u^N$ in fact hold for \emph{every} $t \in (0,T)$. We assumed that the source terms are in $L^\infty(0,T;L^2(\Omega))$,  so their Cl\'ement interpolants are bounded pointwise in time in $L^2(\Omega)$, and therefore $\Phi^N(t)$ and $u^N(t)$ are bounded in various Sobolev spaces for every time. Hence the convergences in \eqref{eq:pointwiseinTimeConv} and the equation for $\Phi(t)$ and the quasi-variational inequality for $u(t)$ that are written in \eqref{eq:pointwiseEquations} also hold for every time. What is crucial is that the $L^\infty$ bound for $\theta_1^N(t)$ given in \eqref{eq:LinftyTheta1NT} is also valid for every $t \in (0,T)$. Hence the argument to show that $\chi^N(t) \to \chi(t)$ in the proof of Proposition \ref{prop:new1} works for every $t \in (0,T)$ with the passage to the limit in $N$ being valid pointwise in time for the source terms and $\theta_i^N$ by the two results given in Lemma \ref{lem:pointwiseConvResults}. This gives us
\[\chi^N(t) \to \chi_{\{u(t)=\Phi(t)\}}\quad \text{in $L^p(\Omega)$ for every $t \in (0,T)$.}\]
Let us now prove that $\chi$ is continuous. Recall 
$\hat L$ from the proof of Proposition \ref{prop:new1}. The idea is to take two times $t, s \in [0,T]$ and once again apply Theorem 4.7 of \cite[\S 5:4]{Rodrigues} for the two obstacle problems satisfied by $u^N(t)$ and $u^N(s)$. Doing so, we obtain, just like \eqref{eq:new1eq},
\begin{align*}
\mu\norm{\chi_{\{\Phi^N(t)=u^N(t)\}}-\chi_{\{\Phi^N(s)=u^N(s)\}}}{L^1(\Omega)}
&\leq \norm{f^N(t)-f^N(s)}{L^1(\Omega)} + a\norm{g^N(t)-g^N(s)}{L^1(\Omega)}\\
&\quad + a\alpha\norm{(\theta_1^N(t)-\theta_2^N(t)-(\theta_1^N(s)-\theta_2^N(s)))\chi^N(t)}{L^1(\Omega)} \\
&\quad+ a\alpha\norm{(\theta_1^N(s)-\theta_2^N(s))(\chi^N(t)-\chi^N(s))}{L^1(\Omega)}\\
&\leq \norm{f^N(t)-f^N(s)}{L^1(\Omega)} + a\norm{g^N(t)-g^N(s)}{L^1(\Omega)}\\
&\quad + a\alpha\norm{(\theta_1^N(t)-\theta_2^N(t))-(\theta_1^N(s)-\theta_2^N(s))}{L^1(\Omega)}\\
&\quad +  a\alpha \hat L\norm{\chi^N(t)-\chi^N(s)}{L^1(\Omega)}.
\end{align*}
Moving the last term on the right-hand side onto the left, taking the limit as $N \to \infty$ and making use of the above-obtained fact that $\chi^N(t) \to \chi(t)$ for every $t \in (0,T)$ and Lemma \ref{lem:pointwiseConvResults}, we get
\begin{align*}
(\mu- a\alpha \hat L)\norm{\chi(t)-\chi(s)}{L^1(\Omega)}
&\leq \norm{f(t)-f(s)}{L^1(\Omega)} + a\norm{g(t)-g(s)}{L^1(\Omega)}\\
&\quad + a\alpha\norm{(\theta_1(t)-\theta_2(t))-(\theta_1(s)-\theta_2(s))}{L^1(\Omega)}.
\end{align*}
Therefore, 
it follows that $\chi \in C^0((0,T);L^1(\Omega))$ because $\theta_i \in C^0([0,T];L^2(\Omega))$ and $f, g \in C^0((0,T);L^1(\Omega))$. Since $\chi$ is a characteristic function, it is also continuous into $L^p(\Omega)$.
\end{proof}
Let us prove the remaining regularity claims of Theorem \ref{thm:continuityParabolic}. Recall from \S \ref{sec:limitingBehaviour} that $\theta_i \in C^0([0,T];Y)$ for any $Y$ such that $H^1(\Omega) \ctsCompact Y \cts L^2(\Omega)$. Taking $Y=L^q(\Omega)$ for $q < 2^*:=2n\slash (n-2)$, we have $\theta_i \in C^0([0,T];L^q(\Omega))$ and hence $(\theta_1-\theta_2)\chi \in C^0((0,T);L^q(\Omega))$ for any $q < 2^*$ because $\chi$ is bounded and is in $C^0((0,T);L^p(\Omega))$ for all $p<\infty$.

This implies, via
\begin{align*}
\norm{-\Delta \Phi(t)+ \Delta \Phi(s)}{L^p(\Omega)} &\leq a\alpha\norm{(\theta_1(t)-\theta_2(t))\chi(t)-(\theta_1(s)-\theta_2(s))\chi(s)}{L^p(\Omega)} + \norm{g(t)-g(s)}{L^p(\Omega)}\\
&\leq a\alpha\norm{(\theta_1(t)-\theta_2(t))-(\theta_1(s)-\theta_2(s))}{L^p(\Omega)} + \norm{(\chi(t)-\chi(s))(\theta_1(s)-\theta_2(s))}{L^p(\Omega)}\\
&\quad + \norm{g(t)-g(s)}{L^p(\Omega)}
\end{align*}
that if $g \in C^0((0,T);L^r(\Omega))$ for $r > 1$, then $-\Delta \Phi \in C^0((0,T);L^{\min(r,q)}(\Omega))$ for all $q < 2^*$. Elliptic regularity then gives $\Phi \in C^0((0,T);W^{2,\min(r,q)}_{\loc}(\Omega))$. That $-\Delta u \in C^0((0,T);L^{\min(r,q)-\epsilon})$ for all $q < 2^*$ now follows directly from the equation \eqref{eq:qLimiting} satisfied by $u(t)$ and the assumption $f \in C^0((0,T);L^r(\Omega))$. 

\subsection{Uniqueness}\label{sec:uniquenessParabolic}
We start with a parabolic version of the continuous dependence result of Proposition \ref{prop:L1ctsDepElliptic} for the temperatures in the elliptic setting. We use again the constants $\gamma_1, \gamma_2$ defined in \eqref{eq:defnOfGammai} (with $\norm{\sigma}{\infty} \equiv 1$).
\begin{prop}[\textsc{$L^1$-continuous dependence}]\label{prop:L1CtsPar}
Let $a' \equiv 0$ and let $(\theta_1, \theta_2, \Phi, u, \chi)$ and $(\hat\theta_1, \hat\theta_2, \hat\Phi, \hat u, \hat\chi)$ denote two (regular) solutions of \eqref{eq:parabolicEquation} corresponding to different data with additionally $h_i, \hat h_i \in L^\infty(Q)$ and $\theta_{i0}, \hat \theta_{i0} \in L^\infty(\Omega)$ for $i=1,2$. Then 
\begin{align*}
\gamma_1\norm{\theta_1-\hat \theta_1}{L^1(0,T;L^1(\Omega))}  +\gamma_2\norm{\theta_2-\hat \theta_2}{L^1(0,T;L^1(\Omega))} &\leq \norm{h_1-\hat h_1}{L^1(0,T;L^1(\Omega))} +\norm{h_2-\hat h_2}{L^1(0,T;L^1(\Omega))}\\
&\quad + \norm{\theta_{10}-\hat \theta_{10}}{L^1(\Omega)} + \norm{\theta_{20}-\hat \theta_{20}}{L^1(\Omega)}\\
&\quad + (L-l)(b_1+b_2)\norm{\chi-\hat\chi}{L^1(0,T;L^1(\Omega))}.
\end{align*}
\end{prop}
\begin{proof}The argument is almost identical to that of Proposition \ref{prop:L1ctsDepElliptic}. Recall the truncation function $T_\epsilon$ from Proposition \ref{prop:L1ctsDepElliptic} and define its antiderivative $S_\epsilon(s) := \int_0^s T_\epsilon(r)\;\mathrm{d}r$. We again test 
\begin{align*}
\partial_t (\theta_i -\hat \theta_i) - \kappa_i \Delta (\theta_i -\hat \theta_i)+ c_i(\theta_i-\hat \theta_i)
 &=h_i - \hat h_i+  (-1)^ib_i((\theta_1-\theta_2)(\chi-\hat \chi) + (\theta_1-\hat\theta_1 + \hat \theta_2 -\theta_2 )\hat \chi)
\end{align*}	
with $\epsilon^{-1}T_\epsilon(\theta_i - \hat\theta_i)$ and doing so, we get, letting the notation $-i$ stand for $1$ when $i=2$ and $2$ when $i=1$,
\begin{align*}
\frac{d}{dt}\frac{1}{\epsilon}\int_\Omega S_\epsilon(\theta_i-\hat\theta_i) + c_i\norm{\theta_i-\hat \theta_i}{L^1(\Omega)}  + b_i\int_{\Omega} |\theta_i-\hat \theta_i|\hat \chi &\leq \norm{h_i - \hat h_i}{L^1(\Omega)} + (L-l)b_i\norm{\chi-\hat\chi}{L^1(\Omega)}\\
&\quad + b_{-i}\int_{\Omega}\hat \chi|\hat \theta_{-i}-\theta_{-i}|.
\end{align*}
Integrating in time and neglecting the term involving $S_\epsilon$ on the left-hand side since $S_\epsilon \geq 0$, using the fact that $\epsilon^{-1}S_\epsilon(s) \to |s|$, and adding for $i=1,2$, we arrive at
\begin{align*}
c_1\norm{\theta_1-\hat \theta_1}{L^1(0,T;L^1(\Omega))}  + c_2\norm{\theta_2-\hat \theta_2}{L^1(0,T;L^1(\Omega))}  &\leq \norm{h_1-\hat h_1}{L^1(0,T;L^1(\Omega))} +\norm{h_2-\hat h_2}{L^1(0,T;L^1(\Omega))}\\
&\quad + \norm{\theta_{10}-\hat \theta_{10}}{L^1(\Omega)} + \norm{\theta_{20}-\hat \theta_{20}}{L^1(\Omega)} \\
&\quad + (L-l)(b_1+b_2)\norm{\chi-\hat\chi}{L^1(0,T;L^1(\Omega))} \\
&\quad + (b_1-b_2)\int_0^T\int_{\Omega}\hat \chi|\hat \theta_2-\theta_2|+ (b_2-b_1)\int_0^T\int_{\Omega} \hat \chi |\hat \theta_1-\theta_1|.
\end{align*}
\end{proof}

We are left to proof the uniqueness of solutions for the evolutionary model.

\begin{proof}[Proof of Theorem \ref{thm:uniquenessPar}]Now let 
 $(\theta_1, \theta_2, \Phi, u, \chi)$ and $(\hat\theta_1, \hat\theta_2, \hat\Phi, \hat u, \hat\chi)$ denote two regular solutions corresponding to the same data. From Proposition \ref{prop:L1CtsPar}, we obtain
\begin{align*}
\norm{\theta_1-\hat \theta_1}{L^1(0,T;L^1(\Omega))} +
\norm{\theta_2-\hat \theta_2}{L^1(0,T;L^1(\Omega))} &\leq \frac{(L-l)(b_1+b_2)}{\gamma_0}\norm{\chi-\hat\chi}{L^1(0,T;L^1(\Omega))}.
\end{align*}
We also obtain in the same fashion as in the proof of Theorem \ref{thm:uniquenessElliptic} (in \S \ref{sec:uniquenessElliptic})
\begin{align*}
\norm{\Delta \hat \Phi - \Delta \Phi}{L^1(0,T;L^1(\Omega))} 
&\leq  \alpha (L-l)\left(\frac{b_1+b_2}{\gamma_0}+ 1\right)\norm{\chi-\hat\chi}{L^1(0,T;L^1(\Omega))}.
\end{align*}
Noting that
\[f+a\Delta\Phi = f-ag-a\alpha(\theta_1-\theta_2)\chi \geq f-ag-a\alpha (L-l),\]
the non-degeneracy condition \eqref{ass:strongNonDegForUniqueness0Par} implies that there exists a constant $\mu$ with 
\[f+a\Delta\Phi \geq \mu > a\alpha (L-l)\left(1 + \frac{b_1+b_2}{\gamma_0}\right).\]
This implies that the non-degeneracy condition of \cite[Theorem 4.7, \S 5:4]{Rodrigues} is valid and it can be applied to yield, after integrating in time and using the above estimate on the obstacles,
\begin{align*}
\mu\norm{\chi-\hat\chi}{L^1(0,T;L^1(\Omega))} &\leq a\norm{\Delta \hat \Phi - \Delta \Phi}{L^1(0,T;L^1(\Omega))} \leq a\alpha (L-l)\left(1 + \frac{b_1+b_2}{\gamma_0}\right)\norm{\chi-\hat\chi}{L^1(0,T;L^1(\Omega))},
\end{align*}
which shows that $\chi = \hat \chi$.
\end{proof}

\subsection{Remarks on weaker solutions}\label{sec:remarksOnNonRegularCase}
Let us discuss the situation where the identification of $\{\chi^k\}$ (recall the semi-discretisation we employed in \eqref{eq:parabolicEquationDiscWeaker}) as characteristic functions from the result of Proposition \ref{prop:regularitySemiDisc} is not available and hence we have at our disposal only the results up to and including \S \ref{sec:limitingBehaviour}. We investigate to what extent we can obtain a weaker existence result for \eqref{eq:parabolicEquation} analogous to Theorem \ref{thm:existence} for the evolutionary model.

The uniform estimates on the interpolants constructed in \S \ref{sec:interpolants} remain in force and hence the Bochner convergences \eqref{eq:listOfBochnerConv} and the pointwise a.e. in time convergences \eqref{eq:pointwiseinTimeConv} are still available, and in addition, we also have (from the $L^\infty(0,T;H^1(\Omega)^*)$ bound in Lemma \ref{lem:thetaINBounds}) for the $t$-dependent subsequence $N_j$, the convergence
\[\partial_t \hat \theta_i^{N_j}(t) \weaklyto \eta_i(t) \qquad \text{in $H^1(\Omega)^*$}\]
to some $\eta_i(t) \in H^1(\Omega)^*$. In general, without the non-degeneracy assumption \eqref{ass:forRegularitySemiDisc} (or \eqref{ass:conditionToGetCharPar}), one \textit{cannot} identify as $\eta_i$ as $\partial_t \theta_i$ from the current information alone. If we pass to the limit in the system \eqref{eq:parabolicNEquations} for the interpolated quantities, we get the existence of what we call a \textbf{very weak solution} of \eqref{eq:parabolicEquation}:
\begin{align*}
&\theta_i \in L^\infty(0,T;H^1(\Omega)) \cap L^2(0,T;H^2_{\loc}(\Omega))\cap W^{1,\infty}(0,T;H^1(\Omega)^*) \cap H^1(0,T;L^2(\Omega)),\\
&\eta_i(t) \in H^1(\Omega)^*,\\
&(u(t),\Phi(t)) \in H^1_0(\Omega) \times (H^2_{\loc}(\Omega)\cap H^1_0(\Omega)),\\
&\tilde \chi \in L^\infty(Q), \text{ $0 \leq \tilde \chi \leq 1$},
\end{align*}
and $\chi(t) \in 1-H(\Phi(t)-u(t))$ for a.e. $t \in (0,T)$ such that
\begin{equation*}
\begin{aligned}
&\text{for $i=1, 2$:} &\eta_i - \kappa_i \Delta \theta_i + c_i\theta_i &= h_i + (-1)^ib_i(\theta_1-\theta_2)\chi &&\text{in $Q$},\\
&&\eta_i &= \partial_t \theta_i - (-1)^ib_i(\theta_1-\theta_2)(\tilde \chi-\chi)\\
&&\partial_n \theta_i &= 0 &&\text{on $\Sigma$},\\
&& \theta_i(0) &= \theta_{i0} &&\text{in $\Omega$},\\
&\text{for a.e. $t \in (0,T)$:}\\
&&-\Delta \Phi(t) &= \alpha(\theta_1(t) - \theta_2(t))\chi(t)+g(t) &&\text{in $\Omega$},\\
&&\Phi(t) &= 0 &&\text{on $\partial\Omega$},\\
&&\hspace{-2.5cm}u(t) \in \mathbb{K}(\Phi(t))\::\: \langle -\grad \cdot (a(\theta_1(t))\grad u(t))-f(t), u(t) -v \rangle &\leq 0 \quad\forall v \in \mathbb{K}(\Phi(t)),\\
&&u(t) &=0 &&\text{on $\partial\Omega$}.
\end{aligned}
\end{equation*}
Observe that the system for $\theta_i$ and $\eta_i$ and the system for $\Phi(t)$ and $u(t)$ are completely decoupled. The desired identification $\eta_i \equiv \partial_t\theta_i$ and $\tilde \chi \equiv \chi$ under these relaxed assumptions appears non-trivial. It is clear that if we had $\partial_t \hat \theta_i^N \to \partial_t \theta_i$ in $C^0([0,T];X)$ for some space $X$ then we could identify $\eta_i \equiv \partial_t \theta_i$ since the pointwise in time limit must agree with $\eta_i(t)$. Even convergence in $L^2(0,T;L^2(\Omega))$ %
would be sufficient for this purpose: it would imply $(\theta_1^N-\theta_2^N)\chi^N \to (\theta_1-\theta_2)\tilde \chi$ and hence $-\Delta\Phi^N \to -\Delta\tilde\Phi$ in the same space, yielding $\Phi^N \to \tilde \Phi$ in $L^2(0,T;H^2_{\loc}(\Omega))$. This would allow us to pass to the limit in the quasi-variational inequality for $u^N$ directly (which we could not do, see the paragraph after \eqref{eq:bochnerEquations}) at least locally, and hence there would be no need to consider the pointwise a.e. in time limits that we were forced to take. 
The missing tool we need is an analogue of the continuous dependence result for the characteristic functions that we used in the non-degenerate case in \S \ref{sec:concExistence}. 

Nonetheless, taking a weighted sum and a weighted difference of the two equations satisfied by $\eta_1$ and $\eta_2$, we find the two relations between $\theta_i$ and $\eta_i$:
\begin{align*}
b_2\partial_t \theta_1+b_1\partial_t \theta_2 &= b_2\eta_1+b_1\eta_2,\\
b_2\partial_t \theta_1 - b_1\partial_t \theta_2 &= b_2\eta_1-b_1\eta_2 - 2b_1b_2(\theta_1-\theta_2)(\tilde \chi - \chi),
\end{align*}
and we deduce that
\[\partial_t \theta_i = \eta_i \text{ on } \{\theta_1 = \theta_2\} \cup \{\tilde \chi = \chi \}.\]
In summary, we are unable to obtain the analogue of Theorem \ref{thm:existence} on the existence of a `weak' time-dependent solution but we can show existence of very weak solutions. The resolution of the issues raised above is an interesting open problem. 

\appendix

\section{Properties of interpolants}\label{sec:app}
\begin{lem}\label{lem:differentInterpolants}The weak limits for the weakly convergent subsequences of $\{\hat \theta_i^N\}$ and $\{\theta_i^N\}$ are the same.
\end{lem}
\begin{proof}
Let us denote by $\theta_i$ the weak limit of $\{\theta_i^N\}$ (we have relabelled the subsequence). We have
\begin{align*}
\norm{\theta_i^N(t) - \hat \theta_i^N(t)}{L^2(\Omega)} \leq \sum_{k=1}^N\norm{\theta_i^k-\theta_i^{k-1} + \frac{\theta_i^{k-1}-\theta_i^k}{\tau}(t-t_{k-1})}{L^2(\Omega)}\chi_{I_k}(t)
\end{align*}
and hence, squaring and using the fact that the $I_k$ are disjoint, we obtain after integrating,
\begin{align*}
\int_0^T\norm{\theta_i^N(t) - \hat \theta_i^N(t)}{L^2(\Omega)}^2 &\leq 2\int_0^T\sum_{k=1}^N\norm{\theta_i^k-\theta_i^{k-1}}{L^2(\Omega)}^2\chi_{I_k}\\
&\leq C\tau\int_0^T\\
&=CT\tau.
\end{align*}
That is, $\theta_i^N - \hat\theta_i^N \to 0$ in $L^2(0,T;L^2(\Omega))$, which along with $\theta_i^N \weaklyto \theta_i$ tells us that $\hat \theta_i^N \weaklyto \theta_i$.
\end{proof}
\begin{lem}\label{lem:pointwiseConvResults}
We have the following pointwise in time convergence results.
\begin{enumerate}[(1)]
\item If $f \in C^{0,\gamma}([0,T];L^1(\Omega))$, then for all $t \in (0,T)$,
\begin{equation*}
f^N(t) \to f(t) \text{ in $L^1(\Omega)$}.
\end{equation*}
\item For all $t \in (0,T)$,
\begin{equation*}
\theta_i^N(t) \to \theta_i(t) \text{ in $L^2(\Omega)$}
\end{equation*}
\end{enumerate}
\end{lem}
\begin{proof}
\begin{enumerate}[(1)]
\item Writing
\begin{align*}
f^N(t)-f(t) &= \sum_{k=1}^N \left(\frac 1\tau \int_{I_k} f(s)-f(t)\;\mathrm{d}s\right)\chi_{I_k}(t),
\end{align*}
the claim follows from
\begin{align*}
\norm{f^N(t)-f(t)}{L^1(\Omega)} &\leq \sum_{k=1}^N \norm{\left(\frac 1\tau \int_{I_k} f(s)-f(t)\;\mathrm{d}s\right)}{L^1(\Omega)}\chi_{I_k}(t)\\
&\leq \frac 1\tau\sum_{k=1}^N \int_\Omega \left|\left(\int_{I_k} f(s)-f(t)\;\mathrm{d}s\right)\right|\chi_{I_k}(t)\\
&\leq \frac 1\tau\sum_{k=1}^N \int_\Omega \left(\int_{I_k} \left|f(s)-f(t)\right|\;\mathrm{d}s\right)\chi_{I_k}(t)\\
&= \frac 1\tau\sum_{k=1}^N \left(\int_{I_k}\int_\Omega  \left|f(s)-f(t)\right|\;\mathrm{d}s\right)\chi_{I_k}(t)\\
&\leq \tau^{\gamma-1}\sum_{k=1}^N \left(\int_{I_k} \;\mathrm{d}s\right)\chi_{I_k}(t)\\
&= \tau^\gamma.
\end{align*}
\item We know that $\hat\theta_i^N \to \theta_i$ in $C^0([0,T];L^2(\Omega))$. 
From the proof of Lemma \ref{lem:differentInterpolants}, we have
\begin{align*}
\norm{\theta_i^N(t) - \hat \theta_i^N(t)}{L^2(\Omega)} &\leq \sum_{k=1}^N\norm{\theta_i^k-\theta_i^{k-1} + \tau^{-1}(\theta_i^{k-1}-\theta_i^k)(t-t_{k-1})}{L^2(\Omega)}\chi_{I_k}(t)\\
&\leq 2\sum_{k=1}^N\norm{\theta_i^k-\theta_i^{k-1}}{L^2(\Omega)}\chi_{I_k}(t)\\
&\leq 2\left(\sum_{k=1}^N\norm{\theta_i^k-\theta_i^{k-1}}{L^2(\Omega)}^2\right)^{1\slash 2}\left( \sum_{k=1}^N\chi_{I_k}(t)^2\right)^{1\slash 2}\tag{by H\"older's inequality}\\
&\leq C\tau^{1\slash 2}
\end{align*}
with the last line by the estimate in Lemma \ref{lem:discBounds}. This allows us to estimate
\begin{align*}
\norm{\theta_i^N(t)-\theta_i(t)}{L^2(\Omega)} &\leq \norm{\theta_i^N(t)-\hat \theta_i^N(t)}{L^2(\Omega)} + \norm{\hat \theta_i^N(t)-\theta_i(t)}{L^2(\Omega)} \to 0 \quad\text{for all  $t \in (0,T)$.}
\end{align*}
\end{enumerate}
\end{proof}

\bibliography{ThermoformingBib}
\bibliographystyle{abbrv}
\end{document}